\documentclass[12]{amsart}

\usepackage{amssymb}
\newtheorem{theorem}{Theorem}

\newtheorem{remark}[theorem]{Remark}

\newtheorem{corollary}[theorem]{Corollary}
\newtheorem{lemma}[theorem]{Lemma}

\newtheorem{definition}[theorem]{Definition}

\newtheorem{proposition}[theorem]{Proposition}

\begin{document} 

\title[ On Engel goups, nilpotent groups, rings and braces]{On Engel groups, nilpotent groups, rings, braces and the Yang-Baxter equation}
\author[Smoktunowicz]{Agata Smoktunowicz}

\address{School of Mathematics,
The University of Edinburgh,
James Clerk Maxwell Building, 
The King's Buildings, 
Peter Guthrie Tait Road, Edinburgh
EH9 3FD, United Kingdom }
\email{A.Smoktunowicz@ed.ac.uk}
\thanks{  This research was supported with ERC advanced grant 320974.}


\date{\today}
\begin{abstract}  
 It is shown that over an arbitrary field there  exists a nil algebra $R$ whose adjoint group $R^{o}$ is not an Engel group. This answers a question by Amberg and Sysak from 1997. The case of an uncountable field  also answers a recent  question by Zelmanov.

 In 2007, Rump introduced braces and radical chains $A^{n+1}=A\cdot A^{n}$ and $A^{(n+1)}=A^{(n)}\cdot A$ of a brace $A$. We show that the adjoint group $A^{o}$ of a finite right brace is a nilpotent group if and only if $A^{(n)}=0$ for some $n$.
 We also show that the adjoint group $A^{o}$ of a finite left brace $A$  is a nilpotent group if and only if  
 $A^{n}=0$ for some $n$. Moreover,  if $A$ is a finite brace whose adjoint group $A^{o}$ is nilpotent then  $A$ is the direct sum of braces whose cardinatities are powers of prime numbers.  Notice that $A^{o}$ is sometimes called the multiplicative group of a brace $A$.  We also introduce a chain of ideals $A^{[n]}$ of a left brace $A$ and then  use it to  investigate braces which satisfy $A^{n}=0$ and $A^{(m)}=0$ for some $m, n$. 

  We also describe connections  between our results and braided groups and  the non-degenerate involutive set-theoretic  solutions of the Yang-Baxter equation.
  It is worth noticing that by a result by Gateva-Ivanova  braces are in one-to-one correspondence with  braided groups with involutive braiding operators.
\end{abstract}

\maketitle

\section{Introduction} 
   In \cite{rump}, Rump  introduced braces as a generalisation of Jacobson radical rings and  as a tool for describing solutions of  the  Yang-Baxter equation.  In the same paper, he introduced 
the following two series of subsets $A^{n}$ and $A^{(n)}$ of a right brace $A$, defined inductively as $A^{n+1}=A\cdot A^{n}$ and $A^{(n+1)}=A^{(n)}\cdot A$, where $A=A^{1}=A^{(1)}$. 
 We will  also use the notation  $A^{n+1}=A\cdot A^{n}$ and $A^{(n+1)}=A^{(n)}\cdot A$ where $A=A^{1}=A^{(1)}$   for a left brace $A$.

  Let $A$ be a finitely generated Jacobson radical ring. It is known that the adjoint group $A^{o}$ of $A$ is a nilpotent group if and only if $A$ is a nilpotent ring, i.e.,  $A^{n}=0$ for some $n$ \cite{ads}. 
In this paper, we show that a similar result holds for finite braces.
\begin{theorem}\label{1}
 Let $A$ be a finite left brace. Then the adjoint group of $A$ is nilpotent if and only if $A^{n}=0$ for some $n$.
 Moreover, such a brace is the direct sum of braces whose cardinalites are powers of prime numbers.
\end{theorem}
Recall that the direct sum $A=\oplus_{i=0}^{n}A_{i}$ of braces is defined  in the same way as for rings; namely if   $a=(a_{1}, \ldots ,a_{n})\in A$ and $b=(b_{1}, \ldots , b_{n})\in A$, then  $a+ b= (a_{1}+b_{1}, \ldots , a_{n}+b_{n})$ and $a\cdot b= (a_{1}\cdot b_{1}, \ldots , a_{n}\cdot b_{n})$. Recall that a  result of Rump shows that if  $A$ is a left brace whose adjoint group $A^{o}$ is a finite $p$-group then $A^{n}=0$ for some $n$ (Corollary after Proposition  $8$, \cite{rump}).  
  Notice that if $A$ is a left brace, then by using the opposite multiplication we get a 
right brace; therefore if $A$ is a right brace, then the group $A^{o}$ is nilpotent if and only if $A^{(n)}=0$. 

Observe that by writing Example $3$ of Rump from \cite{rump} in the language of left braces, we see that there is a  left brace $A$
 of cardinality $6$ such that $A^{(3)}=0$ and  $A^{n}\neq 0$ for every $n$, the adjoint group $A^{o}$ is not a nilpotent group,  and hence $A^{\circ}$ is not an Engel group.
 This shows that the adjoint group of a finite brace need not be a nilpotent group,  and that the assumption of Theorem \ref{1} that $A^{n}=0$ for some $n$  is necessary.
 Notice that  by writing Example $2$ of Rump  from \cite{rump} in the language of  left braces we get that there is a finite left brace $A$  such that $A^{4}=0$ and  $A^{(n)}\neq 0$  for every $n$, whose adjoint group $A^{o}$ is a nilpotent group (\cite{rump}, Example $2$). 

 Recall that in \cite{rump} Rump  introduced 
the following two series of subsets $A^{n}$ and $A^{(n)}$ of a right brace $A$, defined inductively as $A^{n+1}=A\cdot A^{n}$ and $A^{(n+1)}=A^{(n)}\cdot A$, where $A=A^{1}=A^{(1)}$.
 We introduce the following chain $A^{[ n]}$ of ideals of any left or right brace $A$:
 \[A^{[n+1]}=\sum_{i=1}^{n}A^{[i]}\cdot A^{[n+1-i]},\]
  where $A^{[1]}=A$. It is clear that $A^{[n]}\subseteq A^{[n-1]}\subseteq \ldots \subseteq A^{[1]}=A,$ and that for every $i$, $A^{[i]}$ is a two-sided ideal of $A$.
Recall that for subsets $C, D\subseteq A$ we use notation  $C\cdot D=\sum_{i=1}^{\infty }c_{i}d_{i}$ 
 with $c_{i}\in C, d_{i}\in D$ where  almost all $c_{i}, d_{i}$ are zero (so the sums $ \sum_{i=1}^{\infty }c_{i}d_{i}$ are finite).
  Our next results follow.
\begin{theorem}\label{222}
Let $A$ be a left brace (finite or infinite) such that $A^{[s]}=0$ for some $s$.  
If  $a\in A^{[i]}$, $b\in A^{[j]}$, $c\in A^{[k]}$ then 
 \[(a+b)c-ac-bc\in A^{[i+j+k]}.\]   

 Let $P\subseteq  A$  and  let $S$ be the set of all products of elements from $P$. If  $R$ is the additive subgroup of $A$ generated by elements from $S$,  then $R$ is  a left brace (with the addition and the multiplication inherited from $A$). Moreover, if $P$ is a finite set then $R$ is a finite left brace.
\end{theorem}

 We obtain that the following result holds  for both finite and infinite braces.
\begin{theorem}\label{eq23}
  Let $A$ be a left  brace and let $n$ be a natural number. Then the following assertions are equivalent:  
\begin{itemize}
\item[1.] $A^{(n)}=0$  and $A^{m}=0$ for some natural numbers $m,n$.
\item[2.] $A^{[n]}=0$ for some natural number  $n$.
\item[3.] $A^{(n)}=0$ for some $n$ and the group $A^{o}$ is nilpotent.
\item[4.]  The adjoint group $A^{o}$ is nilpotent, and the solution of the Yang-Baxter equation associated to $A$ is a multipermutation solution ($A^{o}$ is also called the multiplicative group of the brace $A$ in \cite{cjo}).
\end{itemize}
\end{theorem}  
Recall that in  \cite{etingof}, Etingof, Shedler and Soloviev introduced a retraction of a solution of the Yang-Baxter equation. A solution $(X,r)$ is called a multipermutation solution of level $m$ 
if $m$ is the smallest nonnegative integer that, after applying the operation of retraction $m$ times,  the obtained solution has cardinality $1$. If such $m$ exists the solution is also called retractable (see \cite{etingof} or \cite{retractable} page $3$ for a more detailed definition). Such a solution is also called a {\em multipermutation solution}, that is a solution which has a finite multipermutation level  (for a detailed definition see \cite{Tatyana}, \cite{cjo}). 
  There are many interesting results in this area \cite{dbs, bc, bcj,  simple, cjo, cjonil,  retractable, etingof, tatyana6, rump, lv, v}. 
  Proposition $5.16$ from \cite{Tatyana},  Proposition $7$ \cite{rump} and the above Theorem \ref{eq23}  motivated the following related result:
\begin{remark}\label{remark}\cite{GIS}
 Let $A$ be a left brace, and let $(A,r)$ be the solution to the Yang-Baxter equation associated to $A$ (as at the beginning of the 
 Section $2$).
  Then $(A,r)$ is  a solution of multipermutation level $m<\infty $ if and only if $A^{(m+1)}=0$ and $A^{(m)}\neq 0$.
\end{remark}
 The proof of Remark \ref{remark}  is very similar to the proof of Proposition $5.16$ \cite{Tatyana} and can be found in \cite{GIS};  it is also possible to prove it by applying  Proposition $7$ \cite{rump} several times translated to left braces.

  Our next result concerns adjoint groups of radical rings and  nil rings. Recall that nil rings have been used by many authors to construct examples of groups; for example triply factorized groups, $SN$-groups, torsion groups, Engel groups and $p$-groups. Therefore, it might be useful to describe new methods for constructing and investigating such rings. This is one of the aims which motivated our next result.

 Recall that if $R$ is any ring then the adjoint semigroup of $R$ is constructed according to the following rule: $a\circ b = ab + a + b$. It is also denoted $1+R$, and it is a group if and only if $R$ is a Jacobson radical ring.  
 Amberg, Catino, Dickenschied, Kazarin, Plotkin, Shalev,  Sysak and others  proved many interesting results on the adjoint group of a radical  ring  \cite{1, 2, 3, 4, 5, plotkin, aw, hal, cjo, shalev}. Amongst many other interesting results, Amberg, Dickenschied and Sysak showed that the adjoint group $R^{o}$ of any Jacobson radical ring  is an SN-group in which
every finite subgroup is nilpotent \cite{ads} (recall that a group G is an SN-group if it has a series with
abelian factors, see \cite{rob}, Vol. 1, pp. 9f and 25). As mentioned in their paper, by 
using Zelmanov's theorem on the restricted Burnside problem (see \cite{z1, z2, z3}), properties of SN-groups and their new ingenious ideas, 
 they were able to deduce the following: 
If $R$ is a finitely generated Jacobson radical ring, then the following are equivalent:
 (a) $R$ is an $n$-Engel ring for some $n\geq 1$
  (b) $R$ is a nilpotent ring (c) $R^{o}$ is an $n$-Engel group for some $n\geq 1$.
 Recall that the aforementioned result of Zelmanov asserts that 
 an n-Engel Lie algebra over an arbitrary field is locally nilpotent, and that any torsion free n-Engel Lie ring is nilpotent \cite{z1, z2, z3}. A surprisingly short proof by  Shalev  assures that if a radical ring $R$ is an $n$-Engel algebra over a field of prime characteristic then the adjoint group $R^{o}$ of $R$ is $m$-Engel for some $m$ \cite{shalev}. A natural question arises then  whether an analogy of any of these results would hold for Engel groups and Engel Lie rings. Notice that every nil ring is an Engel Lie ring (for some interesting related results see \cite{jain, psz, sz, st, ttt, t}). Golod has constructed a nil and not locally nilpotent ring whose adjoint group is an Engel group.
   In 1997 in \cite{ads}, Amberg and Sysak asked  the following question: {\em If $R$ is a nil ring, is the adjoint group $R^{o}$ an Engel group?} Similar questions were also asked in \cite{ads, 2}. At the conference in Porto Cesareo in July 2015, after one of the talks Zelmanov asked the following question: {\em If $R$ is a nil algebra over an uncountable field, is the adjoint group $R^{o}$ an Engel group?} 
 Our result answers these questions in the negative:
\begin{theorem}\label{cesareo}
 There is a nil ring $R$ such that the adjoint group of $R^{o}$ is not an Engel group.
 Moreover, $R$  can be taken to be an algebra over an arbitrary field.
\end{theorem}
 
 The paper is organized as follows:  in Section $2$ we mention connections with the Yang-Baxter equation and braided groups. In Sections $\ref{111}- \ref{k}$ we prove Theorem \ref{cesareo}.  In Sections \ref{p}-\ref{q} we 
  prove Theorems \ref{1},  \ref{222}, \ref{eq23}.   Sections  \ref{111}-\ref{k} and Sections \ref{p}-\ref{q} can be read independently.

\section{Notation and applications for the Yang-Baxter equations and for braided groups}\label{druga}

 Around  2005, Rump introduced braces as a generalisation of Jacobson radical rings. He also showed that
 braces correspond to solutions of the Yang-Baxter equation \cite{rump}. In \cite{LZY}  Lu, Yan and Zhu proposed a general way of constructing set-theoretical solutions of the Yang-Baxter equation using braiding operators on groups. 
 In this paper, by a solution of the Yang-Baxter equation we will  mean non-degenerate involutive set-theoretic solution of the Yang-Baxter equation, as in \cite{cjo}. 

  Let $R$ be a Jacobson radical ring;  then $R$ yields a solution $r: R\times R\rightarrow R\times R$ to the Yang-Baxter equation  
 with the Yang-Baxter operator $r(x,y)=(u,v)$, where  $u=x\cdot y+y$, $v=z\cdot x+x$ and $z$ is the inverse of $u=x\cdot y+y$ in the adjoint group $R^{o}$ of $R$  (so $z\cdot (x\cdot y+y)+z+(x\cdot y+y)=0$). The same holds  when $(R, +, \cdot)$ is a left brace, and this solution is called the solution {\em associated to left brace $R$}, and  will be denoted as $(R,r)$ (for a reference see \cite{cjo, cjr, rump}).

  In \cite{mob} pp 128,  Rump  gave the following definition of a right brace:
`` Let $A$ be an abelian group together with a right distributive multiplication, that is,
 \[(a+b)c=ac+bc\]
 for all $a,b, c\in A$. We call $A$ a {\em brace} if the {\em circle operation} \[a\circ b=ab+a+b\] 
 makes $A$ into a group. This group $A^{o}$ will be called the {\em adjoint group} of a brace $A$.''   In \cite{cjo} Ced{\' o}, Jespers and Okninski wrote the definition of a brace in terms of  operation $o$;  in their paper the adjoint group $A^{o}$ is called the multiplicative group of brace $A$.

 Similarly,  a  left brace is an abelian group  $(A,+)$ together with a left distributive multiplication; that is, 
 $a(b + c) = ab + ac$ such that the circle operation $a\circ b=ab+a+b$ makes $A$ into group. 
 For a left brace $A$,  the associativity of $A^{o}$ is easily seen to be equivalent to the equation
$(ab+a+b)c=a(bc)+ac+bc$. 
 A right brace which is also a left brace is called a two-sided brace; Rump has shown that two-sided braces are exactly Jacobson radical rings.

 In \cite{etingof} Etingof, Shedler and Soloviev introduced a retraction of a solution of the Yang-Baxter equation, and described  some classes of the solutions. They also introduced retractable solutions, which are now also called  multipermutation solutions (see \cite{etingof, cjo, Tatyana}).  Theorem $2$ \cite{cjo} by Ced{\' o}, Jespers and Okninski  assures that: {\em If $G$ is a left brace
 then there exists a solution $(X, r')$ of the Yang-Baxter equation such that the solution $Ret(X, r')$ is isomorphic to the solution associated to the left brace $G$, and moreover, $\mathcal{G}(X,r)$ is isomorphic to the multiplicative group of the left brace $G$. Furthermore, if $G$ is finite then $X$ can be taken a finite set}.

 By writing Example $3$ of Rump from \cite{rump} in the language of left braces, we get that there is a finite left brace $A$  whose adjoint group $A^{o}$ is  the symmetric group $S_{3}$ which is not a nilpotent group (\cite{rump}, Example $3$); moreover  $A^{(3)}=0$ and  $A^{n}\neq 0$ for every $n$. By Remark \ref{remark} the solution to the Yang-Baxter equation associated  to $A$ is a multipermutation solution.
    Observe that  by applying the aforementioned Theorem $2$ from \cite{cjo} to this example  we obtain the following remark.
\begin{remark}(related to Example $3$, \cite{rump}) 
  There is a finite  multipermutation solution $(X,r)$ of the Yang-Baxter solution whose permutation group $\mathcal{G}(X,r)$ of left actions associated with $(X,r)$ is not a nilpotent group.
\end{remark}

 Recall that permutation group $\mathcal{G}(X,r)$ of left actions associated with $(X,r)$ was introduced by Tatyana-Ivanova in \cite{tatyana2} (see also \cite{Tatyana}). 
  By writing Example $2$ from \cite{rump} in the language of a left brace we get that there is a finite left brace $A$ 
  such that $A^{4}=0$, $A^{(n)}\neq 0$  for every $n$, whose multiplicative group is a nilpotent group (\cite{rump}, Example $2$).      
 By Remark \ref{remark} the solution associated to $A$ is not a multipermutation solution. 
This implies, together with Theorem $2$ from \cite{cjo}, the following remark.

\begin{remark} (related to Example $2$, \cite{rump}) There is a finite solution $(X,r)$ to the Yang-Baxter equation, which is not a multipermutation solution, and whose
   permutation group $\mathcal{G}(X,r)$ of left actions associated with $(X,r)$ is a nilpotent group.  
\end{remark}
 
 We get a following related  result for (possibly infinite) braces.
\begin{proposition}\label{ags}
 Let $A$ be a left brace such that the solution of the Yang-Baxter equation associated to $A$ is a multipermutation solution. Then  the adjoint group $A^{o}$ of $A$  is a nilpotent if and only if $A^{n}=0$ for some natural number $n$. 
\end{proposition}
 Gateva-
Ivanova and Van den Bergh \cite{tatyana7} and independently 
  Etingof, Schedler and Soloviev \cite{etingof}  gave a group theoretical interpretation of the
set theoretic involutive non-degenerate solutions of the Yang-Baxter equation.
 Ced{\' o}, Jespers and Okninski \cite{cjo, cedo} asked which groups are multiplicative groups of braces. A similar question in the language of ring theory was asked in \cite{4, 6}. 
In this paper we obtain the following corollary of Theorem \ref{cesareo}, which is related to this question.
\begin{corollary}\label{braces} There is a finitely-generated, two-sided brace whose multiplicative group is a torsion group but is not an Engel group.
\end{corollary}
 By a result of Gateva-Ivanova (see Theorem $3.7$,  \cite{Tatyana}), every brace $G$  can be considered as a braided group with the involutive braided operator. 
 Moreover, by  Proposition $6.2$, \cite{Tatyana}, $G$ is a two-sided brace if and only if the corresponding  braided group satisfies the following identity:

\[c({ }^{(abc)^{-1}}c)=({ }^{b^{-1}}c)({ }^{((a^{b})({ }^{b^{-1}}c))^{-1}}c),\]
 for every $a, b, c\in G$.
By combining the Gateva-Ivanova result with Corollary \ref{braces}, we obtain that:

\begin{corollary}\label{braidedgroup} There is a countable, braided group $(G, \sigma )$  with an involutive braided operator $\sigma $  which  is a torsion-group and not an Engel group. Moreover,  $G$ satisfies a non-trivial identity   
 \[c({ }^{(abc)^{-1}}c)=({ }^{b^{-1}}c)({ }^{((a^{b})({ }^{b^{-1}}c))^{-1}}c),\]
for all $a, b, c\in G$. We use notation $\sigma (a, b)=({ }^ab, a^{b})$.
\end{corollary}
 This shows that infinite braided groups satisfying  non-trivial identities can be quite complicated.
 
 We also get the following result for finite braided groups.

\begin{proposition}\label{takeuchi} Let $G$ be a finite nilpotent group and let $(G, \sigma )$ be symmetric group (in the sense of Takeuchi).  
 Let $(G, +,o)$ be a left brace associated to $(G, \sigma )$ as in Theorem $3.8$ in \cite{Tatyana}. Then $(G, +,o)$ is
 a direct sum of left braces whose cardinalites are powers of prime numbers. These braces correspond to Sylow subgroups of $G$.
\end{proposition}


\section{Braces with $A^{n}=0$ and $A^{(m)}=0$}\label{p}

 In \cite{rump} Rump introduced the following two series of subsets of  any  right brace $A$.
One of the series introduced by Rump is $ \ldots \subseteq  A^{(2)}\subseteq A^{(1)}=A,$ where $A^{(n+1)}= A^{(n)}\cdot A$
 and $A^{(1)}=A$. 
 The other series  introduced by Rump is 
 $\ldots \subseteq A^{2}\subseteq A^{1}=A,$ where $A^{n+1}= A\cdot A^{n}$ and $A^{1}=A$. 
  Following Rump, we will  also use the notation  $A^{n+1}=A\cdot A^{n}$ and $A^{(n+1)}=A^{(n)}\cdot A$ where $A=A^{1}=A^{(1)}$   for a left brace $A$.

 Rump has proved that the series $A^{n}$ of every right brace consists of two-sided ideals \cite{rump}. 
 Similarly, for a left brace $A$, the series $A^{(n)}$ consists of two-sided ideals. Recall that $I$ is an ideal in a brace $A$ if for $i, j\in I$ and $a\in A$ we have $i+j\in I$ and $a\cdot i\in I, i\cdot a\in I$, see \cite{rump}. 

 We propose another series, defined for any left or right brace. This series consists of two-sided ideals in  any left or right brace $A$.
 We define the series
 $\ldots \subseteq A^{[2]} \subseteq A^{[1]}=A,$ where 

\[A^{[n+1]}=\sum_{i=1}^{n}A^{[i]}\cdot A^{[n+1-i]}.\]

 Then it is clear that  $A^{[n]}$ is an ideal in $A$ for every $n$, and $A^{[n+1]}\subseteq A^{[n]}$.

 Recall that for subsets $C, D\subseteq A$ we use notation  \[C\cdot D=\sum_{i=1}^{\infty }c_{i}d_{i}\]
 with $c_{i}\in C, d_{i}\in D$, and almost all $c_{i}, d_{i}$ are zero (so the sums $ \sum_{i=1}^{\infty }c_{i}d_{i}$ are finite).


\begin{theorem}\label{one1}
 Let  $(A, \cdot , +)$ be a left or  right brace. If $m,n$ are natural numbers and $A^{n}=A^{(m)}=0$ then $A^{[s]}=0$ for some number $s$.
\end{theorem}
\begin{proof} We will prove the result in the case when $A$ is a right brace, the case when $A$ is a left brace is done by considering the brace with the opposite multiplication. We will proceed by induction on $n$. If $n=2$ then $0=A^{2}=A\cdot A=A^{(2)}=A^{[2]}$, so the result holds. 
 Suppose that there is a natural number $s_{n,m}$ such that any right brace satysfying $A^{n}=0$ and $A^{(m)}=0$
 satisfies $A^{[s_{n,m}]}=0$. 

Assume now that our brace satisfies $A^{n+1}=0$ and $A^{(m)}=0$.
 Let $p> s_{n,m}\cdot m$, and suppose that $a\in A^{[p]}$. Then $a=\sum_{i}a_{i}b_{i}$ for some $a_{i}, b_{i}\in A$ with $a_{i}\in A^{[p-q_{i}]},$ $b_{i}\in A^{[q_{i}]}$, for some numbers $q_{i}$.
 Observe that if $ q_{i}>s_{n,m}$ then $ b_{i}\in A^{n}$  (by the inductive assumption applied to the brace $A/A^{n}$; this brace is well defined as $A^{n}$ is an ideal in $A$).  In this case we get $a_{i}b_{i}\in A\cdot A^{n}=A^{n+1}=0.$ 
 Therefore $q_{i}\leq s_{n,m}$, as otherwise $a_{i}b_{i}=0$. Consequently we can assume that all $q_{i}\leq s_{n,m}$. For each $i$, we can now write $a_{i}=\sum_{i}a_{i,j}b_{i,j}$, and by the same argument as before, we get that each $b_{i,j}\in A^{[r_{i}]}$ for some $r_{i}\leq s_{n,m}$ (as otherwise $b_{i,j}\in A^{n}$ by the inductive assumption applied to $A/A^{n}$, and so $a_{i,j}b_{i,j}\in A^{n+1}=0$). 
 Observe now that since $A$ is a right brace then \[\sum_{i}a_{i}b_{i}=\sum_{i}(\sum _{j}a_{i,j}b_{i,j})b_{i}=\sum_{i,j}(a_{i,j}b_{i,j})b_{i}.\]
 Continuing in this way we get that $a\in \sum_{c_{1}, \ldots ,c_{m}\in A} ((((A\cdot c_{1})\cdot c_{2})\ldots \cdot c_{m-1})\cdot c_{m})$, and since $A^{(m)}=0$ we get that each $a=0$, so $A^{[p]}=0$. 
\end{proof}

\begin{theorem}\label{two2}
Let  $(A, \cdot , +)$ be either a left brace or a right brace. If  $A^{n}=A^{(m)}=0$  for some natural numbers $m,n$, then the multiplicative group of $A$ is a nilpotent group.
\end{theorem}
\begin{proof} Let $a, b\in A$, then $[a,b]=a\circ b\circ a^{-1}\circ b^{-1}$ where 
$a^{-1}$ and $b^{-1}$ are inverses of $a$ and $b$ respectively in the adjoint group $A^{o}$. We will construct a finite lower central series of $A^{o}$. By Theorem \ref{one1} 
 there is $s$ such that $A^{[s]}=0$.
We proceed by induction on $s$. If $A^{[2]}=0$ then $A$ is commutative so the result holds. Suppose that the result holds for all numbers smaller than $s$; by the inductive assumption
 applied to $A'=A/A^{[s-1]}$ we get  
$[[[[A, A]A]\ldots ]A]\in A^{[s-1]}$ ($m$ brackets for some $m$). Since $A^{[s-1]}$ is in the center of $A$ we get that 
$[[[[A, A]A]\ldots ]A]=0$ ($m+1$ brackets), hence 
$A$ has a finite lower central series.
\end{proof}

\begin{theorem}\label{55} Let $A$ be a left brace such that $A^{(n)}=0$ for some $n$. If the multiplicative group of $A$ is nilpotent then $A^{m}=0$ for some $m$, and hence $A^{[s]}=0$ for some $s$. 
\end{theorem}
\begin{proof} By assumption $A^{(n)}=0$ for some $n$. We can assume that $n$ is minimal possible.
 Let $b\in A^{(n-1)}$, $a\in A$ and let $a^{-1}$ and $b^{-1}$ be the inverses of respectively $a$ and $b$ in the adjoint group $A^{o}$.   
Recall that $A^{o}$ is the group under the circle operation $a\circ b=ab+a+b$. 
 We will show that \[a\circ b \circ a^{-1} \circ b^{-1}=ab.\]

Note that $A^{(n-1)}\subseteq Soc(A)=\{ x\in A\mid x\circ a=x+a\}$. By
[38, Corollary after Proposition 6], $A^{(n-1)}$ is an ideal. Hence
$A^{(n-1)}$ is a normal subgroup of the multiplicative group of the left
brace $A$. Let $b\in A^{(n-1)}$ and $a\in A$.
Since $0=b\circ b^{-1}=b+b^{-1}$, we have that $b^{-1}= -b$.
\begin{eqnarray*}
[a,b]&=&a\circ b\circ a^{-1}\circ b^{-1}\\
&=&a\circ b\circ a^{-1} +b^{-1}\quad (\mbox{since }   a\circ b\circ
a^{-1}\in A^{(n-1)})\\
&=&a\circ (b+ a^{-1})-b\\
&=&a\circ b +a\circ a^{-1}-a-b\\
&=&a\circ b-a-b\\
&=&ab
\end{eqnarray*}

Therefore $[a,b]=a\circ b \circ a^{-1} \circ b^{-1}=ab.$

Since the multiplicative group $A^{o}$ of $A$ is nilpotent we get 
\[[a_{m} [\ldots [a_{2} [a_{1},b_{1}]]]]=0\]
 for some $m$. Therefore $[a_{m} [\ldots [a_{2} [a_{1},b_{1}]]]]=a_{m}(a_{m-1}(\ldots (a_{2}(a_{1}b)))$.
  Consequently $ A(A(\ldots A(A^{(n-1)})))=0$ ($m$ brackets).

We will now apply this result to prove our theorem; we will use induction on $n$ (recall that $n$ is such that  $A^{(n)}=0$).
 For $n=2$ the result holds since $A^{(2)}=A^{2}=A^{[2]}$.
Suppose now that the result holds for all numbers smaller than $n$: so if $B$ is a left brace and $B^{(n-1)}=0$ and the adjoint group of $B^{o}$ is nilpotent then $B^{(n')}=0=B^{[n']}$ for some $n'$.
 
 Recall that $A^{(n-1)}$ is an ideal in $A$, and hence a normal subgroup of $A^{o}$ \cite{rump, cjo}, hence the adjoint group of  brace $A/A^{(n-1)}$ is nilpotent. We can apply the inductive assumption  for the  brace $B'=A/A^{(n-1)}$ and we get that $B^{n'}=0$ hence $A^{n'}=A(A(\ldots A)))\subseteq  A^{(n-1)}$.
 Therefore $A^{m+n'}\subseteq A(A(\ldots A(A^{(n-1)})))=0$. By Theorem \ref{one1} we get that  $A^{[s]}=0$ for some $s$. 
\end{proof}
Let us remark that the first part of the above proof was provided by Ferran Ced{\' o} after 
 reading the original proof in the first version in this manuscript. 
\section{ Structure of left  braces with $A^{n}=0$} 

 In this section we observe some connections between nilpotent braces and nilpotent rings.
  We start with the following lemma.

\begin{lemma}\label{fajny}
 Let $s$ be a natural number and let $A$ be a left brace such that $A^{s}=0$ for some $s$.
 Let $a, b\in A$.  
Define inductively elements $d_{i}=d_{i}(a,b), d_{i}'=d_{i}'(a, b)$  as follows:
$d_{0}=a$, $d_{0}'=b$, and for $i\leq 1$ define $d_{i+1}=d_{i}+d_{i}'$ and $d_{i+1}'=d_{i}d_{i}'$.
 Then for every $c\in A$ we have
\[(a+b)c=ac+bc+\sum _{i=0}^{2s} (-1)^{i+1}((d_{i}d_{i}')c-d_{i}(d_{i}'c)).\]
\end{lemma}
\begin{proof} Observe first that by an inductive argument $d_{i}'\in A^{i}$ for each $i$.
  Observe that for $i\geq 1$ we have 
\[d_{i+1}\cdot c=(d_{i}+d_{i}')\cdot c=((d_{i-1}+d_{i-1}')+d_{i-1}d_{i-1}')\cdot c.\]
 Recall that since $A$ is a left brace then 
\[d_{i+1}c=(d_{i-1}+d_{i-1}'+d_{i-1}d_{i-1}')\cdot c=d_{i-1}c+d_{i-1}'c+d_{i-1}(d_{i-1}'c).\]
 The same holds when we increase $i$ by $1$, hence 
 $d_{i+2}c=d_{i}c+d_{i}'c+d_{i}(d_{i}'c).$
 Subtracting the above equation from the previous one we get
\[d_{i+1}c-d_{i+2}c=(d_{i-1}c-d_{i}c)+e_{i}\]
 where $e_{i}=d_{i-1}'c-d_{i}'c+d_{i-1}(d_{i-1}'c)-d_{i}(d_{i}'c).$
 Observe that 
\[e_{i}=(d_{i-2}d_{i-2}')c-(d_{i-1}d_{i-1}')c+d_{i-1}(d_{i-1}'c)-d_{i}(d_{i}'c).\]
 Therefore, 
\[\sum_{i=1}^{s}e_{2i}=\sum_{i=1}^{s}(d_{2i-2}d_{2i-2}')c-(d_{2i-1}d_{2i-1}')c+d_{2i-1}(d_{2i-1}'c)-d_{2i}(d_{2i}'c)\]
 Notice that if $i\geq s$ then $d_{i}'\in A^{s}=0$. Therefore
$\sum_{i=1}^{s}e_{2i}= (d_{0}d_{0}')c+ q$ where 
\[q=\sum_{i=1}^{s} d_{2i-1}(d_{2i-1}'c)-(d_{2i-1}d_{2i-1}')c -\sum_{i=1}^{s} d_{2i}(d_{2i}'c)-(d_{2i}d_{2i}')c.\]

 Observe now that $d_{i+1}c-d_{i+2}c=(d_{i-1}c-d_{i}c)+e_{i}$
 implies \[\sum_{i=1}^{s}(d_{2i+1}c-d_{2i+2}c)=\sum_{i=1}^{s}(d_{2i-1}c-d_{2i}c)+\sum_{i=1}^{s}e_{2i},\] therefore 
 \[d_{2s+1}c-d_{2s+2}c=d_{1}c-d_{2}c+\sum_{i=1}^{s}e_{2i}.\]
  Observe that $d_{2s+2}=d_{2s+1}+d_{2s+1}'=d_{2s+1}$ since $d_{2s+1}'\in A^{s}=0$.
 Consequently $d_{2}c-d_{1}c= \sum_{i=1}^{s}e_{2i}= (d_{0}d_{0}')c+q.$
 Recall that $d_{0}=a$, $d_{0}'=b$, $d_{1}=a+b$, $d_{1}'=ab$ and $d_{2}=a+b+ab$. Therefore  $d_{1}c=(a+b)c$ and $d_{2}c=(a+b+ab)c=ac+bc +a(bc)$. 
 Consequently $d_{1}c=d_{2}c- (d_{0}d_{0}')c-q.$
 It follows that $(a+b)c=ac+bc+ a(bc)-(d_{0}d_{0}')c-q=ac+bc+d_{0}(d_{0}'c)-(d_{0}d_{0}')c-q$.
 Notice that $d_{0}(d_{0}'c)-(d_{0}d_{0}')c-q=\sum _{i=0}^{2s} (-1)^{i+1}((d_{i}d_{i}')c-d_{i}(d_{i}'c))$, which finishes the proof.
\end{proof}

For an element $a\in A$ and a natural number $i$ by \[i\cdot a\] we will denote the sum of $i$ copies of element $a$ (hence $0\cdot a=0$).

\begin{lemma}\label{8} Let the assumptions and notation be as in Lemma \ref{fajny}. 
  Suppose that there is a natural number $m$ such that  $m\cdot a=m\cdot b=0$. Let $d_{i}, d_{i}'$ be as in Lemma \ref{fajny};
 then $m\cdot d_{i}=m\cdot d'_{i}=0$ for every $i\geq 1$.
\end{lemma}
\begin{proof}  We will first show that $m\cdot d_{t}'=0$ for every $t\geq 0$. For $t=0$ we have $m'\cdot d_{o}'=m\cdot b=0$. Suppose the result holds for some $t\geq 0$, then 
  $m\cdot d_{t+1}'=m\cdot (d_{t}d_{t}')=d_{t}(m\cdot d_{t}')=0$.

 We will now that $m\cdot d_{t}=0$ for all $t\geq 0$. For $t=0$ we have $m\cdot d_{0}=m\cdot a=0$. 
   Suppose the result holds for some $t\geq 0$, then 
  $m\cdot d_{t+1}=m\cdot d_{t}+m\cdot d_{t}'=0$ by the inductive assumption.
 \end{proof}

 Let $A$ be a left brace and let $S, Q\subseteq A$ be additive subgroups of $A$. Then we denote $S+Q=\{s+q: s\in S, q\in Q\}$.   
 \begin{lemma}\label{56}
 Let $(A, +, \cdot )$ be a finite left brace of cardinality 
$p_{1}^{\alpha _{1}}\ldots p_{k}^{\alpha _{k}}$ for some prime pairwise distinct numbers $p_{1}, \ldots , p_{k}$ and natural numbers $\alpha _{1}, \ldots , \alpha _{k}$. Then  \[A=A_{1}+A_{2}+\ldots +A_{k}\] where $A_{i}$ is the additive subgroup  of the additive group  $(A, +)$ of cardinality $p_{i}^{\alpha _{i}}$ for every $i\leq k$. Moreover, $(A_{i}, +,\cdot)$ is a brace for each $i\leq k$. 
\end{lemma}
\begin{proof}
 Since the additive group of $A$ is a finite abelian group,  then using the primary decomposition theorem we can decompose the additive group $(A, +)$  into a  sum of additive subgroups of $A$; we can call them $A_{1}, \ldots , A_{k}$, where $A_{i}$ is an additive subgroup of $A$ of cardinality  $p_{i}^{\alpha _{i}}$ and $A_{i}\cap A_{j}=0$. Observe  that if $x,y\in A$ and $p\cdot y=0$ for some natural number $p$ then $p\cdot (xy)=x \cdot (py)=0$. 
 Therefore if $a, a'\in A_{i}$ then $a\cdot a'\in A_{i}$, hence $A_{i}$ is closed under the multiplication. We know that $A_{i}$ is closed under the addition, hence it is also closed under the operation $\circ $, where $a\circ b=a\cdot b+a+b$ for $a,b\in A$. 
   Observe that since $A$ is a finite group, then the inverse of $a\in A$ in the adjoint group $A^{o}$ is of the form $a\circ a \circ \cdots \circ a$, hence it belongs to $A_{i}$. It follows that $A_{i}$ is a left brace. 
\end{proof}
\begin{theorem}\label{aggi1} 
Let $A$ be a finite left brace such that $A^{n}=0$ for some $n$. Then $A$ is the direct sum of 
 of braces whose cardinalities are powers of prime numbers. In particular, the adjoint group $A^{o}$ of $A$ is a nilpotent group.
\end{theorem}
\begin{proof} Let notation be as in Lemma \ref{56}. 
 We will first show that if $a, b\in A$ and  $m\cdot a=m'\cdot c=0$ for some coprime natural numbers $m,m'$ then $a\cdot c=0$.
 Let $a\in A.$ By $\deg (a)$ we will denote the largest number $i\leq n$ such that $a\in A^{i}$.
We will proceed by induction on $i=2n-\deg (a) -\deg (c)$.   If $2n-\deg (a) -\deg (c)=0$ then $a, c\in A^{n}=0$, so the result holds.
 Suppose now that $i>0$ and that result holds when $2n-\deg (a) -\deg (c)<i$. We will show that the result also holds for $2n-\deg (a) -\deg (c)=i$.

Let $j$ be a natural number. Let $q, d_{1}, d_{1}',  \ldots , d_{n}, d_{n'}$ be as in Lemma  \ref{fajny}, applied for $a$ and for $b=ja$ and for $s=n$. Denote $q_{j}=q$ then by Lemma \ref{fajny} we have  
\[(a+ja)c=ac+(ja)c+q_{j}.\]  
 By Lemma \ref{8}, $m\cdot d_{i}=m\cdot d_{i}'=0$ for all $i\geq 0$. 
 Observe now that for any $i$, the order of element $d_{i}'c$ is a divisor of $m'$ and hence is coprime with $m$. 
 This follows because, by the assumption at the beginning of the proof the order of $c$ is $m'$ and $m'$ is coprime with $m$.
Observe that $m'\cdot (d_{i}'c)=d_{i}'c+\ldots +d_{i}'c=d_{i}'\cdot (m'\cdot c)=0$. 

Observe  $d_{i}(d_{i}'c)=0$ by the inductive assumption, as 
$2n-\deg (d_{i})-deg (d_{i}'c)\leq 2n-\deg (d_{i})-(\deg (c)+1)<2n-\deg (a) -\deg (c)$. Similarly $(d_{i}d_{i}')c=0$ by the inductive assumption since 
 $2n-\deg (d_{i}d_{i}')-deg (c)\leq 2n-(\deg (d_{i}')+1)-\deg (c) < 2n-\deg (a) -\deg (c)$. Therefore $q_{j}=0$.
  Consequently for every natural number $j$, 
\[(a+ja)c=ac+(ja)c.\]

Recall that $m, m'$ are coprime numbers; therefore there are natural numbers $\xi , \beta $ such that 
$\beta m'-\xi  \cdot m=1$. 
 Denote $e=\xi \cdot m+1=\beta \cdot m'$. 
 Observe now that by the above  
 $ac=(ea)c=((e-1)a)c+ ac=((e-2)a)c+ac+ac=\ldots =e(ac)=a(ec)=0$.
 We have proved that $ac=0$. Therefore if $a\in A_{i}$ and $c\in A_{j}$ then $ac=0$, provided that $i\neq j$ ( where $A_{i}$ is as in Lemma \ref{56}).
 
Let $a_{i}\in A^{i}$ for $i=1, \ldots ,k$ and $b\in A$.  By the property of a left brace 
 \[b\cdot (\sum_{i=1}^{k}a_{i})=\sum_{i=1}^{k}ba_{i}.\]

 Let $c_{i}\in A_{i}$. To show that $A$ is the direct sum of braces $A_{i}$ it remains to show that  
 $(\sum _{j=1}^{k}a_{j})c_{i}=a_{i}c_{i}$. 
    We will show  that for every $l\leq k$, $(\sum _{j=1}^{l}a_{j})c_{i}=a_{i}c_{i}$  if $i\leq l$ and $(\sum _{j=1}^{l}a_{j})c_{i}=0$ if $i>l$.  
  We will proceed by induction on $l$. The result  is true for $l=1$. Let $l>1$ and suppose that the result holds $l-1$.

 Observe first that $a_{l}\cdot (\sum_{j=1}^{l-1}a_{j})=\sum_{j=1}^{l-1}a_{l}a_{j}=0$ by the first part of the proof.
Hence $(\sum _{j=1}^{l}a_{j})c_{i}=(a_{l}+(\sum_{j=1}^{l-1}a_{j})+a_{l}(\sum_{j=1}^{l-1}a_{j}))c_{i}=
 a_{l}c_{i}+(\sum_{j=1}^{l-1}a_{j})c_{i}+ a_{l}((\sum_{j=1}^{l-1}a_{j})c_{i})$.
 By the inductive assumption $(\sum_{j=1}^{l-1}a_{j})c_{i}=a_{i}c_{i}$ if $i\leq l-1$ and $(\sum_{j=1}^{l-1}a_{j})c_{i}=0$ otherwise. 
 Suppose that $i>l$ then  $(\sum_{j=1}^{l-1}a_{j})c_{i}=0$ and $a_{l}c_{i}=0$ hence $(\sum_{j=1}^{l}a_{j})c_{i}=0,$ as required.
 If $i=l$ then $(\sum_{j=1}^{l-1}a_{j})c_{i}=0$ so  $(\sum_{j=1}^{l}a_{j})c_{i}=a_{l}c_{i}=a_{i}c_{i}$ as required.
 If $i<l$ then $(\sum_{j=1}^{l-1}a_{j})c_{i}=a_{i}c_{i}$, $a_{l}c_{i}=0$ and  $(a_{l}(\sum_{j=1}^{l-1}a_{j})c_{i}))=a_{l}(a_{i}c_{i})=0$
 as $a_{l}\in A_{l}$ and $a_{i}c_{i}\in A_{i}$ and $l\neq i$. Hence $(\sum_{j=1}^{l}a_{j})c_{i}=a_{l}c_{i}=a_{i}c_{i}$ as required.
 Therefore, $A$ is the direct sum of braces $A_{i}$.

We will now show the nilpotency of $A$. Observe first that for every $i$, $A_{i}$ is a $p$-group and hence is nilpotent.
 Observe then that if $a\in A_{i}$ and $b\in A_{j}$ for $i\neq j$ then $a\circ b=b \circ a$ since $a\circ b=a+b+ab=a+b$ and 
 $b\circ a=b+a+ba=b+a$ by the above. Therefore, $A^{o}$ is the direct product of groups $A_{i}$ for $i=1, \ldots ,k$, and hence it is  a nilpotent group. 
\end{proof}
\section{Braces whose adjoint group is nilpotent}\label{q}

In this section we will investigate the structure of braces whose adjoint groups are nilpotent. 
  For the following result we use a short proof which was provided by Ferran Ced{\' o} after  
 reading the original proof in the first version in this manuscript.  
\begin{theorem}\label{aggi3}
 Let $A$ be a finite left brace such that the adjoint group  $A^{o}$ is a nilpotent group. Then $A$ 
 is a direct sum of braces whose cardinalities are powers of prime numbers.
 Assume that  $A$ has cardinality $p_{1}^{\alpha _{1}}\cdots p_{k}^{\alpha _{k}}$, for some prime pairwise distinct numbers $p_{1}, \ldots , p_{k}$ and some natural numbers $\alpha _{1}, \ldots , \alpha _{k}$. Then $A^{n}=0$ where 
$n$ is the largest number from among $\alpha _{1}+1$, $\alpha _{2}+1, \ldots ,\alpha _{k}+1$.
\end{theorem}
\begin{proof} (Provided by Ferran Ced{\' o}.) The first part is easier to prove using the equivalent definition of 
left brace. [\cite{cjo}, Definition 1]:  A left brace is a set $B$ with two binary 
operations: a sum $+$ and a multiplication $\circ$, such that $(B,+)$ is 
an abelian group, $(B,\circ )$ is a group and
$a\circ (b+c)+a=a\circ b+a\circ c$ for all $a,b,c\in B$.

Suppose that $B$ is a finite left brace such that its multiplicative 
group is nilpotent. Let $P$ be a Sylow $p$-subgroup of the additive 
group of the left brace $B$. By [\cite{cjo}, Lemma $1$], $\lambda_a(P)=P$ for all 
$a\in B$, where $\lambda_a(b)=a\circ b-a$. In particular $P$ is closed 
by the multiplication and hence it is a subgroup of the multiplicative 
group of the left brace $B$. Thus $P$ is a Sylow $p$-subgroup of the 
multiplicative group of $B$. Since the multiplicative group of $B$ is 
nilpotent, $P$ is a normal subgroup in $(B,\circ)$. Hence $P$ is an 
ideal of the left brace $B$ (see [\cite{cjo}, Definition 3]). Therefore, if 
$P_1,\dots ,P_r$ are the Sylow subgroups of the additive group of  $B$, 
then they are also the Sylow subgroups of the multiplicative group of 
$B$, in fact they are ideals of $B$ and $B=P_1\circ\dots\circ 
P_r=P_1+\dots +P_r$ is the inner direct product of the subbraces 
$P_1,\dots ,P_r$.

The second part of Theorem 19 is a consequence of [\cite{rump}, Corollary after 
Proposition 8].
\end{proof}
{\bf Proof of Theorem \ref{1}} If $A^{n}=0$ for some $n$, then by Theorem \ref{aggi1} the group $A^{o}$ is nilpotent, and $A$ is the direct sum of braces whose cardinalities are prime numbers. 
On the other hand if $A$ is a left brace and $A^{o}$ is nilpotent then $A^{n}=0$ for some $n$, by Theorem \ref{aggi3}.

{\bf Proof of Theorem \ref{222}} This follows from Lemma \ref{fajny} applied several times, taking into account that $A^{[s]}=0$.

{\bf Proof of Theorem \ref{eq23}} 
 Notice that $1$ and $2$ are equivalent by Theorem \ref{one1}. 
 Notice that by Remark \ref{remark}, $3$ and $4$ are equivalent. By Theorems \ref{two2} and \ref{55} properties  $3$ and $1$ are equivalent.
  
{\bf Proof of  Proposition \ref{ags}} This follows from Remark \ref{remark} and Theorems \ref{two2}, \ref{55}.


\section{Jacobson radical}\label{111}

In this chapter we give some preliminary results on Jacobson radical rings.

\begin{lemma}\label{pierwszy} Let $F$ be a field.  Let $n$ be a natural number. 
 Let $R$ be an $F$-algebra generated by elements $a,b$ (without an identity element), and suppose that $a^{2}=0$ and $b^{n}=0$ for some $n$.
 Let $S$ be the $F$-linear space spanned by elements $a\cdot b^{i}$ for $0<i<n.$
 If all finite  matrices  with entries from $S$
 are nilpotent, then $R$ is a Jacobson radical  ring.
\end{lemma}
\begin{proof}
 We will use the well-known fact that  a one-sided ideal in which every element is quasi-regular generates a two-sided ideal which is Jacobson radical \cite{lam}.
Let $R'$ be a subring of $R$ generated by elements from $S$. 
Since all matrices with entries from $S$ are nilpotent, then by Theorem $1.2$ from  \cite{smok}
  $R'$ is a Jacobson radical ring. 
Consider ring $S'$  generated by elements from $S$ and from $Sa$ and by element $a$. Recall that  $a^{2}=0$, and so $SaS=0$. Therefore $Sa$ is a two sided-ideal in $S'$ which is nilpotent; also  $S'/Sa$ is Jacobson radical, since $R'$ is Jacobson radical. It follows that  $S'$ is Jacobson radical.

Observe that $S'R\subseteq S'+S'a=S'$, hence $S'$ is a right ideal in $R$.   Therefore the two sided ideal generated by $S'$ in $R$  is Jacobson radical; we will call this ideal $I$. 
Observe now that the ring $R/I$ is nilpotent, as it is generated by powers of $b$. It follows that $R$ is Jacobson radical.
\end{proof}

\begin{lemma}\label{trzeci} Let $F$ be a field.  Let $n$ be a natural number. 
 Let $R$ be an $F$-algebra generated by elements $a,b$ (without an identity element), and suppose that $a^{2}=0$ and $b^{n}=0$ for some $n$.
Let $R[x]$ be the polynomial ring in one variable $x$ over $R$. 
Let $Q$ be the $F$-linear space spanned by elements $a\cdot b^{i}x^{j}$ for $0<i<n,$ $0\leq j$.
If all finite  matrices  with entries from $Q$
 are nilpotent, then $R[x]$ is a Jacobson radical  ring, and hence $R$ is a nil ring.
\end{lemma}
\begin{proof}
 Amitsur's theorem assures that if $R$ is a ring such that $R[x]$ is Jacobson radical then $R$ is a nil ring ( Theorem 15A.5, \cite{rowen}).
  Therefore  it suffices to show that $R[x]$ is Jacobson radical.
 Observe that by Theorem $1.2$ from \cite{smok},  if $R'$ is a subring of $R[x]$ generated by elements from $Q$ then 
 $R'$ is Jacobson radical. Let $S'=R'+R'a+F[a]$ (where $F[a]$ denotes the subalgebra of $R$ generated by $a$); then similarly as in Lemma \ref{pierwszy} we get that $S'$ is Jacobson radical.
 It then follows that the two-sided ideal $I$ generated by $S'$ in $R$ is Jacobson radical, and moreover $R[x]/I$ is nil. Therefore 
$R[x]$ is Jacobson radical.
\end{proof}

By $R^{1}$ we denote the usual extension of a ring $R$ 
 by the identity element.

\begin{lemma}\label{two} Let $F$ be a field, and $R=F[a, b]$ be the free algebra (without an identity element) generated by elements $a, b.$  Given $c\in R$, by $F[c]$ we will denote the subalgebra of $R$ generated by $c.$  Let $S$ be the linear $F$-subspace of $F[a,b]$ spanned by elements $ab$ and $a\cdot b^{2}$.
Let  $I$ be the ideal of $F[a,b]$
 generated by $a^{2}, b^{3}$ and by elements from sets $F_{1}, F_{2}, \ldots $ such that $F_{i}\subseteq S^{i}$ for every $i$.  
\begin{itemize}
\item[1.] If $p+q+t+t'+t''\in I$ and $p\in ab R^{1} , q\in bR^{1}aR^{1}$, $t\in F[a]$, $t'\in F[b],$ $t''\in a^{2}R^{1}bR^{1}$ then $p,q,t,t', t''\in I$.
\item[2.]  If $p+q\in I$ and $p\in R^{1} ba, q\in R b$ then $p,q\in I$.
\item[3.] If $p=e_{1}+e_{2}+\ldots +e_{n}$ with $e_{i}\in F\cdot S^{i}$ and $p\in I$ then $e_{i}\in I$ for all $i\leq n$.
\end{itemize}
\end{lemma}
\begin{proof}  

 {\bf 1.}  By specialising at $b=0$ we get that $t\in F[a^{2}]\subseteq I$, and by specialising at $a=0$ we get $t'\subseteq F[b^{3}]\in I$; notice that $t''\in I$, hence $p+q\in I$. 
  Denote $Z=\bigcup _{i=1}^{\infty } F_{i}$. Notice that 
 $I\subseteq Z R^{1}+b^{3} R^{1}+a^{2} R^{1}+bI+aI$. It follows that 
 $p+q\in   (ZR^{1}+aI+a^{2}R^{1})+(bI+b^{3}R^{1})$.  Observe that $Z\subseteq aR$.
 Therefore, $p\in (ZR^{1}+aI+a^{2}R^{1})\subseteq I$ and $q\in (bI+b^{3}R^{1})\subseteq I$, as required.

{\bf 2.} Observe now that $I\subseteq Ia+Ib+R^{1}\cdot Z+R^{1} a^{2}+R^{1} b^{3}$, hence $p+q\in (Ia \cap R^{1}ba)+Rb+R^{1} a^{2}$. Notice that $p\in R^{1}ba$ and  $q\in Rb$; 
it follows that $p\in Ia\cap R^{1}ba$, and hence $p\in I$ and so $q=(p+q)-p\in I$.

{\bf 3.} Let $J$ be the ideal of $R$ generated by elements from sets $F_{i}$, and let $<a^{2}>,$ $ <b^{3}>$ denote ideals generated by $a^{2}$ and $b^{3}$ respectively; then 
  $p-j\in <a^{2}>+<b^{3}>$ for some $j\in J$.  Notice that $p$ has no terms from $<a^{2}>+<b^{3}>$.
 Consequently  $p\in I'+bI'+b^{2}I'+I'a+bI'a+b^{2}I'a$, where $I'$ is the ideal of $E$ generated by elements from sets $F_{i}$ and $F_{i}'$, where $E$ is the $F$-algebra generated by elements from $S$, and where  $F_{i}'= (F_{i}b+F_{i}b^{2}+Rb^{3})\cap E$.
  Hence $p-i\in bI' +b^{2}I'+I'a+bI'a+b^{2}I'a$ for some $i\in I'$. Since the left hand side belongs to $aR^{1}b$ and the right 
 hand side to $Ra+bR$ it follows that $p-i=0,$ so $p\in I'\subseteq E$. Therefore in the factor ring $E/I'$ we have that $p+I'$ is the zero element.

Observe that $I'$ is an homogeneous ideal in $E$ when we assign gradation of elements from $S$ to have gradation $1$, from $S^{2}$ gradation $2$ etc. Now $p+I'=0$ in $E/I'$, so 
$\sum_{j=1}^{n}(e_{j}+I')=0$ in $E/I'$, and since $E/I'$ is a graded ring and each $e_{j}+I'$ has gradation $j$
it follows that $e_{i}+I'=0$, hence $e_{i}\in I$, for every $I$. 
\end{proof}

\section{ Ideals generated by powers of matrices are `small`}

Let $R$ be a ring and $R[x]$ be the polynomial ring over $R$. 
Given a matrix $M$ with entries from $R[x]$, 
 let $P(M)$ denote the linear space spanned by coefficients of polynomials which are entries of matrix $M$.

 We will say that a ring $R$ and a linear space $S$  satisfy {\bf Assumption $1$} when 
\begin{itemize}
\item[1.] $R$ is the free algebra (without identity) generated by two elements $a, b$ over a field $F$.
\item[2.]   $S$ is the linear $F$-subspace of $F[a,b]$ spanned by elements $ab$ and $a\cdot b^{2}$.
\end{itemize}

\begin{lemma}\label{podslowa}
 Let $R$, $S$ satisfy assumption $1$, and let $R[x]$ be the polynomial ring over $R$ in one variable $x$.  Let $m$ be a natural number and let $M$ be a matrix with entries from $S^{m}\cdot F[x]$. Let $C=\{c_{1}, c_{2}, \ldots , c_{j}\}$, where $c_{1}, \ldots , c_{j}$ are nonzero elements from $F\cdot S^{m}$.  Let $r=r_{1}r_{2}r_{3}$ where $r_{i}$ is a product of $n_{i}$ elements from set $C$, for $i=1,2, 3$, with $n_{i}\geq 0$.
 
If $r\in P(M^{n_{1}+n_{2}+n_{3}})$ then $r_{i}\in  P(M^{n_{i}})$ for $i=1,2,3$.

\end{lemma}
\begin{proof}
 We can write $M^{n}=M^{n_{1}}\cdot M^{n_{2}}\cdot M^{n_{3}}$. Therefore $P(M^{n})\subseteq P(M^{n_{1}})\cdot P(M^{n_{2}})\cdot P(M^{n_{3}})$.
 Hence \[r=r_{1}\cdot r_{2}\cdot r_{3}\in P(M^{n_{1}})P(M^{n_{2}})P(M^{n_{3}})\]
 and $r_{i}\in F\cdot S^{m\cdot n_{i}},  P(M^{n_{i}})\subseteq F\cdot S^{m\cdot n_{i}}$; it follows that $r_{i}\in P(M^{n_{i}})$ for $i=1,2,3$.

Indeed, if $r_{j}\notin P(M^{n_{j}})$ for some $j$, then we would find a linear mapping $f: F\cdot S^{m \cdot n_{j}}\rightarrow 
 F\cdot S^{m\cdot n_{j}}$ such that $f(P(M^{n_j}))=0$ and $f(r_{j})\neq 0$, and we can apply this mapping to the above inclusion at appropriate places, obtaining a contradiction.
\end{proof}

\begin{definition}\label{wazna1} Let $F$ be an infinite field. Let $R, S$ satisfy Assumption $1$. 
Let $f:F\cdot S^{m}\rightarrow F$ be a $F$-linear mapping.
For every $i$  we can extend the mapping $f$ to the mapping $f:F\cdot S^{m\cdot i}\rightarrow F$
 by defining $f(w_{1}\cdots w_{i})=f(w_{1})\cdots f(w_{i})$  for $w_{1}, \ldots , w_{i}\in  S^{m},$ and then extending it by linearity to all elements from $F\cdot S^{m\cdot i}$.

Let $t(x)=\sum_{i=0}^{n}t_{i}x^{i}$ for some $t_{i}\in R$, then we denote $f(t(x))=\sum_{i=0}^{n}f(t_{i})x^{i}.$ Let $M$ be a matrix with entries $m_{i,j}$; by $f(M)$ we will denote the matrix 
with corresponding  entries equal to $f(m_{i,j}).$

Similarly, if $g: F\cdot S^{m}\rightarrow F\cdot S^{m} $ is a linear mapping then 
for every $i$  we can extend the mapping $g$ to the mapping $g: F\cdot S^{m\cdot i}\rightarrow F\cdot S^{m\cdot i}$
 by defining $g(w_{1}\cdots w_{i})=g(w_{1})\cdots g(w_{i})$  for $w_{1}, \ldots , w_{i}\in  S^{m}$ and then extending it by linearity to all elements from $F\cdot S^{m\cdot i}$.

\end{definition}

\begin{lemma}\label{mily} Let notation be as in Lemma \ref{podslowa}. Assume that $f(c_{i})\neq 0$ for all $i\leq j$, where $c_{i}\in S^{m}$ are as in Lemma \ref{podslowa}.  Let $f: F\cdot S^{m}\rightarrow F$ be a linear mapping, and $f:F\cdot S^{m\cdot n_{2}}\rightarrow F$ be as in Definition \ref{wazna1}.
  Let $F$ be an infinite field, and let $n=n_{1}+n_{2}+n_{3}$ be natural numbers. 
If $r=r_{1}r_{2}r_{3}\in P(M^{n})$ then  
\[r_{1}f(r_{2})r_{3}\in P(M^{n_{1}}f(M^{n_{2}})M^{n_{3}}).\]
\end{lemma} 
\begin{proof}
 Let $M$ be a $d$-by-$d$ matrix and let $a_{i,j}(x)$ be the polynomial which is at the $i,j$ entry of $M^{n}$.
Notice that $a_{i,j}(x)=\sum_{k,l\leq d} b_{i,k}(x)c_{k,l}(x)d_{l,j}(x)$, where $b_{i,k}$ is the $i,k$ entry of matrix $M^{n_{1}}$,  
  $c_{k,l}(x)$ is the $k,l$ entry of $M^{n_{2}}$ and $d_{l,j}$ is the $l,j$ entry of matrix $M^{n_{3}}$.
Similarly $n_{i,j}(x)=\sum_{k,l\leq d} b_{i,k}(x)f(c_{k,l}(x))d_{l,j}(x)$ is the $i,j$-th entry of matrix $M^{n_{1}}f(M^{n_{2}})M^{n_{3}}.$

 Notice that since $F$ is infinite, then  by a Vandermonde matrix argument we get that $P(M^{n})=\sum _{i,j\leq d, p\in F}F\cdot a_{i,j}(p)$ and
 \[P(M^{n_{1}}f(M^{n_{2}})M^{n_{3}})=\sum _{i,j\leq d, p\in F}F\cdot n_{i,j}(p)\]

 If $r=r_{1}r_{2}r_{3}\in P(M^{n})$ then $r\in \sum_{i,j\leq D, p\in F}F\cdot  a_{i,j}(p)$ hence 
\[r_{1}r_{2}r_{3}\in span _{ p\in F , i,j\leq d} \sum_{k,l\leq d}b_{i,k}(p)c_{k,l}(p)d_{l,j}(p).\]

If we apply the mapping $f$ as in Definition \ref{wazna1}  at appropriate places we get that 
\[r_{1}f(r_{2})r_{3}\in span _{ p\in F , i,j\leq d} \sum_{k,l\leq d} b_{i,k}(p)f(c_{k,l}(p))d_{l,j}(p).\]

 Recall that $n_{i,j}(x)=\sum_{k,l\leq d} b_{i,k}(x)f(c_{k,l}(x))d_{l,j}(x)$ is the $i,j$-th entry of matrix $M^{n_{1}}f(M^{n_{2}})M^{n_{3}}.$
 Therefore the linear space spanned by elements \[\sum_{k,l\leq d} b_{i,k}(p)f(c_{k,l}(p))d_{l,j}(p)\] for $p\in F$ equals the space spanned by 
$n_{i,j}(p)$ for $p\in P$.
 We have shown at the beginning of this proof that the latter space equals  $P(M^{n_{1}}f(M^{n_{2}})M^{n_{3}}).$ Therefore $r_{1}f(r_{2})r_{3}\in span _{i,j\leq d, p\in F} n_{i,j}(p)\subseteq P(M^{n_{1}}f(M^{n_{2}})M^{n_{3}}).$ 
\end{proof}
\begin{lemma}\label{poly} Let $R$ be an $F$-algebra. Let $P$ be a linear space spanned by the coefficients of polynomials $h_{i}(x)\in R[x]$ for $i=1,2,\ldots $.
Then for arbitrary non-zero polynomial $g(x)$ from $F[x]$ the linear space $Q$ spanned by the coefficients of polynomials $g(x)h_{i}(x)$ equals $P.$
\end{lemma}
\begin{proof} Clearly $P\subseteq Q$.
Let $P_{i}$ denote the space spanned by the coefficients of $x^{i}$ of  polynomials $h_{1}(x), h_{2}(x), \ldots.$
 
By calculating the coefficient by the smallest power of $x$ in polynomials $g(x)h_{i}(x)$ we get that $P_{0}\subseteq Q$. By then calculating the coefficient by the second-smallest power of $x$ in $g(x)h_{i}(x)$ we get that $P_{1}\in Q+P_{0}\subseteq Q$.
 Continuing in this way we get $P_{i}\subseteq  Q+P_{0}+\ldots +P_{i-1}$,  
so $P_{i}\subseteq Q$ for every $i$. It follows that $P\subseteq Q$.

\end{proof}

\begin{lemma}\label{piaty} Let notation be as in Definition \ref{wazna1} and Lemma \ref{mily}. Let $t$ be a natural number and let $M$ be a $d$-by-$d$ matrix. Let $n_{1}, n_{2}, n_{3}\geq 0$. Then for all $i\geq 1$ we have
\[P(M^{n_{1}}f(M^{n_{2}})M^{n_{3}})\subseteq \sum _{i=1}^{d+1}P(M^{n_{1}}f(M)^{i}M^{n_{3}}).\]
  
\end{lemma}
\begin{proof} Let $b_{i,k}$ denote the $i,k$ entry of matrix $M^{n_{1}}$,  
  $c_{k,l}(x)$ denote the $k,l$ entry of $M^{n_{2}}$ and $d_{l,j}$ denote the $l,j$ entry of matrix $M^{n_{3}}$.
 
Let $n_{i,j}(x)$ be the $i,j$-th entry of matrix $M^{n_{1}}f(M^{n_{2}})M^{n_{3}}$, then
 \[n_{i,j}(x)=\sum_{k,l\leq d} b_{i,k}(x)f(c_{k,l}(x))d_{l,j}(x).\]

 Notice that $f(M)$ is a matrix with coefficients from $F[x].$
 Every matrix with entries from  the field of rational functions $F\{x\}$ in variable $x$ satisfies its characteristic polynomial.
  It follows that there are polynomials $f_{i}(x)$ such that \[\sum_{i=1}^{d+1}f_{i}(x)f(M)^{i}=0\] with $f_{d+1}(x)$ nonzero.

 Therefore for every $n$ there is a polynomial  $g_{n}(x)\in F[x]$ such that $g_{n}(x)f(M)^{n}\in \sum_{i=1}^{d+1}F[x]\cdot f(M)^{i}$. By Lemma \ref{poly}, \[P(M^{n_{1}}f(M^{n_{2}})M^{n_{3}})=P(g(x)M^{n_{1}}f(M^{n_{2}})M^{n_{3}})=P(M^{n_{1}}g(x)f(M^{n_{2}})M^{n_{3}}).\]
 We know that $g_{n}(x)f(M^{n_{2}})\subseteq \sum_{i=1}^{d+1}F[x]f(M)^{i}$, hence \[P(M^{n_{1}}g(x)f(M^{n_{2}})M^{n_{3}})\subseteq \sum_{i=1}^{d+1}P(F[x]\cdot M^{n_{1}}f(M)^{i}M^{n_{3}}).\]
    By Lemma \ref{poly} we get $P(M^{n_{1}}f(M^{n_{2}})M^{n_{3}})\subseteq \sum _{i=1}^{d+1}P(M^{n_{1}}f(M)^{i}M^{n_{3}})$.
\end{proof}

We will say that $M, R, S, m, d,\alpha $ satisfy {\bf Assumption 2} if 
\begin{itemize}
\item[1.]  $R, S$ satisfy Assumption $1$ and $m, d, \alpha $ are natural numbers.
\item[2.]  $M$ is a $d$-by-$d$ matrix with entries from $S^{m}\cdot F[x]$.
 Moreover, \[M\subseteq R+Rx+Rx^{2}+\ldots Rx^{\alpha }.\] 
\end{itemize}
 Let $c_{1}, \ldots , c_{j}$ be linearly independent  elements from $F\cdot S^{m}$ and denote $C=\{c_{1}, \ldots , c_{j}\}$.
 Let $v=c_{i_{1}}\ldots c_{i_{j}}$ and $v'=c_{k_{1}}\ldots c_{k_{j}}$ for some $i_{1}, \ldots i_{j}$, and some $k_{1}, \ldots , k_{j}$. We will say that words $v$ and $v'$ are distinct if $i_{l}\neq j_{l}$ for some $l\leq j$.

Let $r$ be a product of elements from the set $C$.
We say that $w$ is a subword of degree $n$ in $r$ if $w$ is a product of $n$ elements from $C$ and $r=vwv'$ for some $v,v'$ which  are also products of elements from $C$.
\begin{lemma}\label{naj} Let $F$ be an infinite field.  $M, R, S, m, d, \alpha $ satisfy Assumption $2$. Let $q$ be natural number.
 Let $c_{1}, \ldots , c_{j}$ be linearly independent  elements from $F\cdot S^{m}$, and let $r,r'$ be products of 
$q$ elements from the set $C=\{c_{1}\ldots , c_{j}\}$. 
 If $n\geq 8d^{3}\cdot (\alpha +1)$ and $r$ has at least $n$ pairwise distinct  subwords of length $n$,
and $r'$ has at least  $n$ pairwise distinct subwords of length $n$  then  $r\cdot r'\notin P(M^{t})$, for any $t$. 
\end{lemma}
\begin{proof} Suppose on the contrary that $rr'\in P(M^{t})$. 
Let $p_{1}, \ldots , p_{n}$ be subwords of $r$ of degree $n$, and $q_{1},q_{2}, \ldots ,q_{n}$ be subwords of $r'$ of length $n$. Then there are $s_{i,k}$  such that 
$p_{i}s_{i,k}q_{k}$ is a subword of $r\cdot  r'$ for all $i,k\leq n$. 
 By Lemma \ref{podslowa}, $r\cdot r'\in P(M^{t})$ implies $p_{i}s_{i,k}q_{k}\in P(M^{m_{i,k}})$ for some $m_{i,k}$. 
 Let $f: F\cdot S^{m}\rightarrow F$ be a linear mapping, and $f:F\cdot S^{m\cdot n_{2}}\rightarrow F$ be as in Definition \ref{wazna1}. We can choose $f$ that $f(c_{i})\neq 0$ for $i=1,2, \ldots ,j$ and hence  $f(s_{i,k})\neq 0$ for every $i,k\leq n$.  
By Lemma \ref{mily},    
\[p_{i}f(s_{i,k})q_{k}\in P(M^{n}f(M^{m_{i,k}-2n})M^{n}).\] By Lemma \ref{piaty}, $p_{i}q_{k}\in \sum_{l=1}^{d+1} P(M^{n}f(M^{l})M^{n})$. Notice that the linear space $\sum _{l=1}^{d+1}P(M^{n}f(M^{l})M^{n})$ has dimension smaller than $d^{2}(d+1)\cdot (2n\alpha +2)$.

Observe now that since $p_{i}$ and $q_{i}$ are products of $n$ elements $c_{i}$, and $c_{i}'s$, and  are  linearly independent over $F$, then elements $p_{i}q_{k}$ are linearly independent over $F$. Therefore 
 elements $p_{i}q_{k}$ span a linear space over field $F$ of dimension at least $n^{2}$. 
Hence $n^{2}\leq  d^{2}(d+1)\cdot  (2n\alpha+2)<8d^{3}\cdot (\alpha +1)\cdot n,$ a contradiction.
\end{proof}

\section{ Subspaces $E_{i}$ and $E_{i}'$}

Let $F$ be a countable field and let $R, S$  satisfy Assumption $1$. 
 Since $F$ is countable, we can ennumerate finite  matrices with entries in $S\cdot F[x] $ as 
   $X_{1}, X_{2}, \ldots $. We can assume that the matrix $X_{i}$ is a $d_{i}$-by-$d_{i}$ matrix where $d_{i}\leq i$ and $X_{i}$ had entries in $F\cdot( S+ Sx+S\cdot x^{2}+\ldots +S\cdot x^{i})$ for every $i$, if necessary taking $X_{i}=0$ for some $i$.

The following is similar to Theorem $5$ from \cite{amitsur}.
 
\begin{theorem}\label{id} Let $F$ be a countable field, let $R, S$ satisfy Assumption $1$ and let matrices $X_{1}, X_{2}, \ldots $ be as above.  Let $0<m_{1}<m_{2}< \ldots $ be a sequence of natural numbers such that $m_{i}$ is a power of two and  $2^{2^{m_{i}}}$ divides $m_{i+1}$ for all $i\geq 1$. Denote $R(m)=F\cdot S^{m}$ for every $m$. 
Let $E'_{i}$ be the linear space spanned by all coefficients of polynomials which are entries of the matrix $X_{i}^{2^{2^{m_{i}}}}$ and  let
\[E_{i}=\sum_{j=0}^{\infty } R(j\cdot m_{i+1})E'_{i}SR.\]
  Then there is an ideal $I$ in $R$ contained in $\sum_{i=1}^{\infty }E_{i}+bE_{i} +b^{2}E_{i}+<a^{2}>+<b^{3}>$   and such that 
$R/I$ is a nil ring, where $<a^{2}>, <b^{3}>$ denote ideals in $R$ generated by elements $a^{2}$ and $b^{3}$.  
\end{theorem}
\begin{proof}  Observe first that the ideal $I_{k}$ of $R$ generated by coefficients of polynomials which are entries of the matrices  
$X_{k}^{2m_{k+1}+2}$  is contained in the subspace $E_{k}+bE_{k}+b^{2}E_{k}$.
  It follows because entries of every matrix $X_{k}$ have degree one in the subring generated by $S$ with elements of $S$ of degree one. In general if $n>m_{k+1}+2^{2^{m_{k}}}+1$ then  coefficients of polynomials which are entries of  matrix
 $X_{k}^{n}$ belong to $R(i)E'_{k}R(1)R$ for every $0\leq i<n-m_{k+1}-1$.
 
Define $I=\sum_{i=1}^{\infty }I_{k}+<a^{2}>+<b^{3}>$; then 
$I\subseteq \sum_{i=1}^{\infty }E_{i}+bE_{i} +b^{2}E_{i}+<a^{2}>+<b^{3}>.$
 Observe also that, by Lemma \ref{trzeci}, $R/I$ is a nil ring.
\end{proof}

\begin{lemma}\label{specialmappings}
Let $F$ be an infinite field and let $T\subseteq L$ be  finitely dimensional $F$-linear spaces.
Let $c_{1}, c_{2}, \ldots c_{j}\in L$ and $c_{1}, c_{2}, \ldots c_{j}\notin T$.
 Then there is a linear mapping $f:L\rightarrow F$ such that $T$ is contained in the kernel of $f$ and  
$c_{1}, c_{2}, \ldots c_{j}$ are not contained in the kernel of $f$.

Moreover there is a linear mapping $g:L\rightarrow L $ such that $T$ is contained in the kernel of $g$ and  
$c_{1}, c_{2}, \ldots c_{j}$ are not contained in the kernel of $g$ and the image of $g$ has co-dimension $1$ in $L$.
\end{lemma}
\begin{proof} Let $Q$ be a maximal linear space such that $c_{1}, c_{2}, \ldots c_{m}\notin Q$ and $T\subseteq Q$. We will show that $L/Q$ is a one dimensional linear space. Suppose on the contrary, then there are two elements $x+Q, y+Q$ in $L/Q$ which are linearly independent over $F$. 
 By maximality of $Q$, we get that there are $\alpha \neq \beta $ such that linear spaces  
 $Q+F\cdot (x+\alpha \cdot y)$ and $Q+F\cdot (x+\beta \cdot y)$  both contain some element $c_{i}$. 
  Then $c_{i}-t_{1}\cdot  (x+\alpha \cdot y)\in Q$ and  $c_{i}-t_{2}\cdot  (x+\beta \cdot y)\in Q$ for some $t_{1}, t_{2}\in F$. It follows that  $t_{1}\cdot  (x+\alpha \cdot y)-t_{2}\cdot (x+\beta \cdot y)\in Q$, 
 a contradiction since $x+Q$ and $y+Q$ are linearly independent in $L/Q$. 
We can now take $g:L\rightarrow L$ to be such that $f(q)=0$ for every $q\in Q$, and the image of $g$ has co-dimension one 
 in $L$. Observe that the natural linear mapping $f: L\rightarrow  L/Q$; then $L/Q$ has dimension $1$, so $L/Q$ is isomorphic as a linear space to $F$. In this way we can define the mapping $f$.
\end{proof} 

\section{Words $w_{i}$}

 In this chapter we give some supporting results on some monomials related to Engel elements.
\begin{definition}\label{words}
Let $M$ be the free monoid generated by elements $A,B , A', B'$. Define inductively a sequence of  infinite monomials
$W_{i}$ as follows. 
\[ W_{1}=A, W_{2}= ABA', W_{3}=W_{2}B{\bar W_{2}}\]
and 
 \[ W_{n+1}=W_{n}B{\bar {W_{n}}},\]
where ${\bar W}=(\sigma (W))^{op}$ where $\sigma :M\rightarrow M$ is a homomorphism of monoids such that $\sigma (A)=A',$ $\sigma (A')=A$, $\sigma (B)=B'$ and $\sigma (B')=B$. 
 Recall that if $x_{i}\in \{A,A', B,B'\}$ then $(x_{1}x_{2}\ldots x_{n})^{op}=x_{n}x_{n-1}\ldots x_{1}$.
 We will sometimes refer to monomials from $M$ as words. 
 Observe also that for every $n>0$,
 \[\bar { W}_{n+1}=W_{n}B'{\bar {W_{n}}}.\]
\end{definition}

\begin{lemma}\label{one}
 For every $i$, $W_{i}$ has length $2^{i}-1$.
\end{lemma} 
\begin{proof} By induction on $i$.
\end{proof}

\begin{lemma}\label{dobry} Let notation be as in Lemma \ref{one}. 
  Let $\alpha (c, v)$ denote the number of occurences of $c$ in a word $v\in M$. Then for every $n\geq 1$ $\alpha (A,W_{n+1})=\alpha (A',W_{n+1})=\alpha (B,W_{n+1})=2^{n-1}$ and $\alpha (B',W_{n+1})=2^{n-1}-1.$ 
Moreover in the word $W_{n+1}$ after elements $A,A'$  appears either $B$ or $B'$; and after elements 
$B,B'$ appears either $A$ or $A'$.
\end{lemma}
\begin{proof} Observe that for $v= W_{n}B'$ we have $\alpha (A, v)+\alpha (A', v)=\alpha (B, v)+\alpha (B', v),$  as in $v$ after $A$ and $A'$ always appears either $B$ or $B'$ and vice-versa. 
 We will now proceed with the proof of our theorem by induction on $n$. For $n=1$ we have $W_{2}=ABA'$, so the result holds in this case. Suppose that the result holds for some $n\geq 1$; we will show that it holds for $n+1$.   

Observe that for every word $v,$ we have  $\alpha (A, v)=\alpha (A', \bar {v})$ and  $\alpha (B, v)=\alpha (B', \bar {v})$.  
 Recall that 
$W_{n+1}=W_{n}B{\bar W}_{n}$, consequently $\alpha (A,W_{n+1})=\alpha (A', W_{n+1}),$ $\alpha (B,W_{n+1})=\alpha (B', W_{n+1})+1$. Let $v=W_{n+1}B'$; then by the above $\alpha (A,v)= \alpha (A',v)$ and $\alpha (B,v)= \alpha (B',v)$. 

 By Lemma \ref{one}, $\alpha (A,v)+\alpha (A',v)+\alpha (B,v)+\alpha (B',v)=2^{n+1}$. Hence $\alpha (A,v)=\alpha (A',v)=\alpha (B,v)=\alpha (B',v)=2^{n-1}.$ 
 The result follows.
\end{proof}

\begin{definition}  Let $M$ be a free  monoid generated by elements $A,A',B,B'$. Let $v\in M$.
 Let $R$ be a ring, and let $x,y,z,t\in R$. By $v(x,y,z,t)$ we will denote the element of $R$ obtained by substituting $A=x$, $B=y$, $A'=z$ and $B'=t$ in the word $v$.
\end{definition}

\begin{definition}\label{zi} Let $F$ be a field, and let  $R$ be a ring generated by elements $a,b$, such that $a^{2}=0, b^{3}=0$. Then 
  $1+a$ and $1+b$ are invertible elements in $R^{1}$.  Denote
 \[v_{1}=[1+a, 1+b]=(1+a)(1+b)(1+a)^{-1}(1+b)^{-1},\]
 \[v_{2}=[v_{1}, 1+b]=v_{1}(1+b)v_{1}^{-1}(1+b)^{-1}, \]
\[v_{n+1}=[v_{n}, 1+b]=v_{n}(1+b)v_{n}^{-1}(1+b)^{-1}.\] 
\end{definition}

\begin{lemma}\label{z} Let notation be as in Definition \ref{zi}.  Denote $z_{n+1}=v_{n}\cdot (1+b)$.
 Then $z_{2}= (1+a)(1+b)(1+a)^{-1}=(1+a)(1+b)(1-a)$
and \[z_{n+1}=z_{n}\cdot (1+b)\cdot z_{n}^{-1}\] for $n=2, 3, \ldots $. 
 Moreover, \[(z_{n+1})^{-1}=z_{n}\cdot (1-b+b^{2})\cdot z_{n}^{-1}.\]
\end{lemma}
\begin{proof} It is clear that $z_{2}=v_{1}(1+a)= (1+a)(1+b)(1+a)^{-1}=(1+a)(1+b)(1-a)$
 By the definition of $v_{n+1}$ we get $z_{n+1}=v_{n}(1+b)=v_{n-1}\cdot (1+b)\cdot v_{n-1}^{-1}=[v_{n-1}(1+b)]\cdot (1+b)\cdot [(1+b)^{-1}v_{n-1}^{-1}]=z_{n}\cdot (1+b)\cdot z_{n}^{-1}$.
 This  implies $z_{n+1}^{-1}=z_{n}(1-b+b^{2})z_{n}^{-1}.$
\end{proof}

{\bf Notation.} Let notation be as in Definition \ref{words}. Let $R$ be a ring generated by elements $a,b$ such that $a^{2}=0$ and $b^{3}=0$.
In what follows we will use the following notation. 
\[w_{n}=W_{n}(a,b, -a, b^{2}-b), \bar {w}_{n}=\bar {W}_{n}(a,b, -a, b^{2}-b).\]

\begin{lemma}\label{w}  Let notation be as above and as in Definition \ref{words}. Then
$w_{1}=a$, $w_{2}=-aba$,
$\bar {w}_{1}=-a$ and $\bar {w}_{2}=-a(b^{2}-b)a.$ Moreover, for every $n$, 
 $w_{n+1}=w_{n}\cdot b\cdot \bar {w}_{n}$ and $ \bar {w}_{n+1}=w_{n}\cdot (b^{2}-b)\cdot \bar {w}_{n}.$
\end{lemma}
\begin{proof} 
 It follows from Definition \ref{words} by induction on $n$. 
\end{proof}

\begin{lemma}\label{important}  Let $F$ be a field and let $R, S$ satisfy Assumption $1$. Let notation be as in Lemmas \ref{w} and  \ref{z}. 
 Let $T(j)$ be the linear space spanned by all monomials  $x_{1}x_{2}\ldots x_{n}$ such that $x_{i}\in \{a,b\}$ and the cardinality of the set $\{1\leq i\leq n-1 :x_{i}x_{i+1}\in \{ab, ba\}\}$ is at most $j$ ($n$ is arbitrary).

Then for every $n\geq 2$, 
\[z_{n}-w_{n}-1, z_{n}^{-1}-\bar {w}_{n}-1\in T(2^{n}-3)\]
and  \[w_{n}, \bar {w}_{n} \in F\cdot S^{2^{n-1}-1}a \subseteq T(2^{n}-2).\]
 Recall that $S$ is the linear space over $F$  spanned by elements $ab$ and $a\cdot b^{2}$.
\end{lemma}
\begin{proof} 
We will proceed by induction. We will use Lemmas \ref{w} and \ref{z}.
 
For $n=2$ we have $w_{2}=-aba\subseteq F\cdot Sa\subseteq T(2)=T(2^{2}-2),$ and 
$z_{2}=(1+a)(1+b)(1-a)=-aba+(ab+b-a^{2}-ba)+1$.
 Therefore,
 $z_{2}-w_{2}-1= ab+b-a^{2}-ba\in T(1)=T(2^{2}-3)$, as required.

 Recall that $\bar {w}_{2}=-a(b^{2}-b)a\in F\cdot Sa\subseteq T(2)$.
 Observe also that $z_{2}^{-1}=(1+a)(1-b+b^{2})(1-a)=-a(b^{2}-b)a+a(b^{2}-b)+(b^{2}-b)-a^{2}-(b^{2}-b)a+1$
hence $z_{2}^{-1}-\bar {w}_{2}-1\in T(1)$, as required.

Suppose now that the result holds for some number $n\geq 2$; we will prove it for $n+1$.
 Observe that  for all $i,j$, we have $T(i)T(j)\subseteq T(i+j+1)$, as we have only one more place when the words from $T(i)$ and $T(j)$ meet where a change from $a$ to $b$ or from $b$ to $a$ can appear.
 
Observe that $w_{n+1}=w_{n}b\bar {w}_{n},$
 hence by the inductive assumption $w_{n+1}\in T(2^{n}-2)T(0)T(2^{n}-2)\subseteq T(2^{n}-2+0+2^{n}-2+2)=T(2^{n+1}-2)$.

 We have $z_{n+1}= z_{n}(1+b)z_{n}^{-1},$ hence for some $q, q'\in T(2^{n}-3)+F$
\[z_{n+1}=(w_{n}+q)(1+b)(\bar {w}_{n}+q').\]
 By the inductive assumption 
$q(1+b)q'\subseteq T(2^{n}-3)T(0)T(2^{n}-3)\subseteq  T(2^{n}-3+0+2^{n}-3+2)=T(2^{n+1}-4)$.
Similarly 
$q(1+b)\bar {w}_{n}\subseteq T(2^{n}-3)T(0)T(2^{n}-2)\subseteq T(2^{n+1}-3)$
 and $w_{n}(1+b)q'\subseteq T(2^{n+1}-3)$. Consequently, $z_{n+1}-w_{n+1}-1\in T(2^{n+1}-3)$
 The proof that $\bar {w}_{n+1}\in T(2^{n+1}-2)$ and $z_{n+1}^{-1}-\bar {w}_{n+1}\in T(2^{n+1}-3)$ is analogous.

Notice also that by Lemma \ref{dobry} and by Notation before Lemma \ref{w} 
\[w_{n}, \bar {w}_{n} \in S^{2^{n-1}-1}a\subseteq T(2^{n}-2).\]
\end{proof}

\begin{lemma}\label{wazny} Let $F$ be an infinite field.  
 Let $R, S$ satisfy Assumption $1$. 
 Let $I$ be the ideal of $ R$ generated by  elements  from sets $F_{1}, F_{2}, \ldots ,$ where $F_{i}\subseteq F\cdot S^{i}$ for every $i$ and by elements $a^{2}$ and $b^{3}$. Denote $w_{n+1}'=W_{n+1}(a, b, -a, b^{2}-b)\in R$, $w_{n+2}'=W_{n+2}(a, b, -a, b^{2}-b)\in R$.
 Let $v_{n}, z_{n+1}$ be as in Definition \ref{zi} and Lemma \ref{z} applied for ring $ \bar {R}=R/\langle a^{2}, b^{3}\rangle $.
 Let  $I'$ be the ideal of $\bar {R}$ which is generated by images in $R/\langle a^{2}, b^{3}\rangle $ of elements  from sets $F_{1}, F_{2}, \ldots .$ 

 If  $v_{n}-1\in I'$ for some $n>1$ then $w_{n+1}'b\in I$, and hence $w_{n+2}'\in I$.
\end{lemma}
\begin{proof}  Let $z_{n+1}'\in R$ be such that 
 the image of $z_{n+1}'$ in $R/\langle a^{2}, b^{3}\rangle $ is $z_{n+1}$. Recall that, $z_{n+1}=v_{n}\cdot (1+b)$ ( as in  Lemma \ref{z}).  Observe that  $v_{n}-1\in I'$  is equivalent to 
 $z_{n+1}\cdot (1+b)^{-1}-1\in I'$. This implies $z_{n+1}-(1+b)\in I'$ and hence
$z_{n+1}'-(b+1)\in I$ (since $a^{2}, b^{3}\in I$). Observe that $R/I'={\bar R}/I$. 
We can write $z_{n+1}'-(b+1)=p+q+t+t'+t''\in I$ for some 
 $p\in ab\cdot R^{1}$, $q\in bR^{1}aR^{1}$, $t\in F[a], t'\in F[b]$, $t''\in a^{2}R^{1}bR^{1}$, as in Lemma \ref{two}.  By Lemma \ref{two} [1.] we get  $p,q,t,t', t''\in I$. 
By Lemma \ref{important} we get $p=w_{n+1}'+v+i$ for some $v\in T(2^{n+1}-3)\cap ab R^{1}$
 and $i\in \langle a^{2}, b^{3}\rangle$.
 We can write $v=v'+v''+v'''$ where $v'\in Rba$, $v''\in Rb$ and $v'''\in Ra^{2}$ (we can assume that $v, v',v''\in abR^{1}$ since $v\in abR^{1}$). Notice that $v'''\in I$ and  $i\in I$, hence $w'_{n+1}+v'+v''\in I$. By Lemma \ref{two} [2.] we get that $ w'_{n+1}+v'\in I$. 
 Therefore,  $w'_{n+1}b+v'b\in I$ and $v'b\in T(2^{n+1}-2).$ Notice that $v'b\in ab\cdot R^{1}\cap R^{1}\cdot ab$.
 It follows that $v'b\in F+ F\cdot S+F\cdot S^{2}+\ldots + F\cdot S^{2^{n}-1}+i'$ for some $i'\in \langle a^{2}, b^{3}\rangle$.
 By Lemma \ref{dobry}, $w_{n+1}'b\subseteq F\cdot S^{2^{n}}$.
 By Lemma \ref{two} [3.]  $w_{n+1}'b\in I$. Observe also that 
 $w_{n+1}'b\in I$ implies $w_{n+2}'=w_{n+1}'b\bar {w}_{n+1}'\in I$.
\end{proof}

\section{Combinatorics of words}

 The following lemma immediately follows from the proof of Theorem 1.3.13, pp.22, \cite{lothaire}. We repeat a slightly modified proof.
\begin{lemma}\label{lothaire1} Let $n$ be a natural number. Let $w$ be an infinite word 
 which has less than $n$ subwords of degree $n$; then $u=cddd\ldots $ for some words $c,d$ such that  
$c$ has length smaller than $n!$ and $d$ has length $n!$. 

 Moreover if  $u$ is a finite word which has less than $n$ subwords of degree $n$ then $u=cddd\ldots d$ for some words $c,d$ such that  
$c$ has length smaller than $n!+n!$ and $d$ has length $n!$. 
 
\end{lemma}

\begin{proof} Notice that for some $m\leq n$, $w$ has the same number of words of length $m$ and $m+1$. Hence for every subword $v_{1}$ in $w$ of length $m$ we 
 have exactly one possibility of a subword $u_{1}$ of $w$ which has  length $m+1$ and which contains $v$ as the beginning.
 Let $v_{2}$  be the ending of length $m$ in $u_{1}$. We can apply the same reasoning to $v_{2}$ instead of $v_{1}$ and find word $u_{2}$ containing $v_{2}$ as the beginning. After at most $n$ steps we get $v_{i}=v_{j}$ for some $i\neq j$ and then from the step $j$
  $v_{k+j}=v_{k}$ for each $k$. Therefore $w=v_{1}\ldots v_{i-1}vvvvvvv...$ where $v$ is some word of length $t<n+1$. 
 Because $t$ divides $n!$ we get the result for $w$.

If  $u$  is a  finite word then we can apply  similar reasoning.
\end{proof}

 We can then get the following.
\begin{lemma}\label{lothaire}
 Let $R, S$ satisfy Assumption $1$. Let $v'=c_{1}\ldots c_{m}$ with each $c_{i}\in F\cdot S^{m}$ for some $m$. We say that $w$ is a subword of $v'$ of length $n$ if 
 $v'=uwu'$ where for some $k$,  $u=c_{1}\cdots c_{k}$, $w=c_{k+1}c_{k+2}\cdots c_{k+n}$, $u'=c_{k+n+1}\cdots c_{m}$ ($u$ and $u'$ may be trivial words).
 Let $n$ be a natural number. Assume that $v'$ has less than $n$ subwords of degree $n$ then $u=cddd\ldots d$ for some subwords $c,d$ such that  
$c$ has length smaller than $n!+n!$ and $d$ has length $n!$. 
\end{lemma}
\begin{proof} It follows from Lemma \ref{lothaire1}.
\end{proof}

\begin{theorem}\label{ia} Let $0<m_{1}< m_{2}< \ldots $ be such that each $m_{i}$ is a power of two, $2^{2^{m_{i}}}<m_{i+1}$ and $2^{2^{m_{i}}}$ divides $m_{i+1}$ for every $i$.  
 Let $R(i), S, E_{1}$ be as in Theorem \ref{id}. Assume that $2^{m_{1}}>17!\cdot 10$.
Then $w_{n}b\notin E_{1}$ for any $n$, where $w_{n}=W_{n}(a, b, -a, b^{2}-b)$. 
\end{theorem}
\begin{proof}  Suppose, on the contrary, that  $w_{n}=W_{n}(a,b,-a,b^{2}-b)b\in E_{1}$.  
 We can assume that $n>m_{2}$, since $w_{i}b\in E_{1}$ implies $w_{i+1}b\in E_{1}$, by the definition of $W_{i}$ and $E_{1}$.
 Recall that $R(2^{n-1})=F\cdot S^{2^{n-1}}$.
Since  $w_{n}b\in R(2^{n-1})$, by Lemma \ref{one} and $m_{2}$ is a power of two, we can assume that $2^{n-1}$ is divisible by $m_{2}$ (because we can take larger $n$ if needed by the argument from the first lines of this proof).

Denote $w=w_{n}b$. We can write $w=w'_{1}\ldots w'_{t}$ with each $w'_{i}\in R(m_{2})$.  Since $w\in E_{1}$ it follows that some $w'_{j}\in E_{1}\cap  R(m_{2})$ (it can be proved using linear mappings similarly as in the proof  of  Lemma \ref{podslowa}). Then $w'_{j}=v_{1}\ldots v_{m_{2}/m_{1}}$ for $v_{i}\in R(m_{1})$. Denote $\alpha = 2^{2^{m_{1}}}/m_{1}$, then 
$v_{1}\ldots v_{\alpha }\in E'_{1}$, where $E'_{1}$ is the linear space spanned by all coefficients which are entries of the matrix $X_{1}^{2^{2^{m_{1}}}}$, by Theorem \ref{id}.

 Write $v_{1}\ldots v_{\alpha }=vv'$, where $v,v'\in R(\alpha /2)$. Recall that $vv'\in P({X_{1}}^{j})$ for $j=2^{2^{m_{1}}}$.
 By Lemma \ref{naj} we get that for every   $n>8d^{3}\cdot (\alpha +1)$ either $v$ or $v'$ has less than $n$ subwords of length $n$; without restraining generality we can assume that it happens for $v'$ (by a subword we mean a product $v_{i}v_{i+1}\cdots v_{j}$ for some $i\leq j$).

 In our case  $\alpha =1$, $d=1$ because of assumptions on matrix $M=X_{1}$ (before Theorem \ref{id}), so we can take $n=17$. 
 By Lemma \ref{lothaire} $v'=cddd\ldots d$ for some words $c,d$ such that  
$c$ has length smaller than $17!+17!$ and $d$ has length $17!$.

 By assumption $2^{2^{m_{1}}}=2^{k}$ for some $k$. By the definition of words $W_{i}$ we get that $vv'=w_{k+1}b$
 and $vv'=w_{k}b\bar {w}_{k}b$. It follows that $v=w_{k}b$ and $v'=\bar {w}_{k}b$.

Observe now that $v'=\bar {w}_{k}b$ implies $v'=w_{k-1}(b^{2}-b)\bar {w}_{k-1}b$.

 By Lemma \ref{dobry} we can write $w_{k-1}=ab_{2^{k-2}-1}\cdot b_{2^{k-2}-1}a\cdots ab_{2}\cdot ab_{1}a$ where $b_{i}\in \{b, b^{2}-b\}$.
 Then $v'=w_{k-1}(b^{2}-b)\bar {w}_{k-1}b$ yields
\[v'=(ab_{2^{k-2}-1}\cdot ab_{2^{k-2}-2}\cdots ab_{2}\cdot ab_{1}\cdot a(b^{2}-b))(a\bar {b}_{1}\cdot a\bar {b}_{2}\cdots a\bar {b}_{2^{k-2}-1}\cdot ab),\] where $\bar {b}=b^{2}-b$ and $\bar {b^{2}-b}=b$.  
  By Lemma \ref{lothaire} applied to $v'$ and to $m=1$ we get that  $v'=cddd\ldots d$ for some words $c\in R(\alpha ),d\in R(17!)$ where  
$c$ has length smaller than $17!+17!$, so $\alpha <17!+17!$.

 By the assumptions of our theorem $k>17!\cdot 10$. 
 We can write
\[v'= {w}_{k-1}(b^{2}-b){\bar {w}_{k-1}}b=s\cdot p \cdot q \cdot r \cdot t\] for some $s,t\in R$ and some $p, q,r\in R(17!)$ with  $spq=w_{k-1}(b^{2}-b)$. Notice that then $s\in R(q)$ for some $q>17!\cdot 2$. Then because $v'=cddd\ldots $ and $d$ has length $17!$ we get $p=q=r$. Recall that $spq=w_{k-1}(b^{2}-b)$, and by the above 
 \[s\cdot p\cdot q=(ab_{2^{k-2}-1}\cdot ab_{2^{k-2}-2}\cdots ab_{2}\cdot ab_{1}\cdot a(b^{2}-b)).\]
 Consequently, $q=ab_{17!-1}a\ldots b_{2}ab_{1}a(b^{2}-b)$, $r=a\bar {b}_{1}a\bar {b}_{2}a\ldots a\bar {b}_{17!}$, and $p=p'ab_{17!}$,  for some $p'$.  Recall that  $p=r$; it follows that 
$\bar {b}_{17!}=b_{17!}$, which is impossible, as $b_{17!}\in \{b, b^{2}-b\}$ and  $\bar {b}=b^{2}-b$ and $\bar {b^{2}-b}=b$. We have obtained a contradiction. 
 \end{proof}
\section{ Mapping $T$}

  In this section we will use the following notation 
 \[w_{t}=W_{t}(a, b, -a, b^{2}-b),  \bar {w}_{t}=\bar {W}_{t}(a, b, -a, b^{2}-b).\] 
 Recall that  by Lemma \ref{one} we have  \[w_{k}b, w_{k}(b^{2}-b)\in R(2^{k-1}).\]
 Recall that, given matrix $M$ with entries in $R[x]$, by $P(M)$ we denote the linear space spanned by coefficients of polynomials which are entries of matrix $M$. The linear spaces $E_{i}$ and $E_{i}'$ are as in Theorem \ref{id}.

 Let $R, S$ satisfy Assumption $1$, by $S$-monomial we will mean a product of elements from the set  $\{ab, a(b^{2}-b)\}$. 

\begin{lemma}\label{t} Let $F$ be a field and let $R, S$ satisfy Assumption $1$. Let $n, j$ be  natural numbers. Let $d_{1}, \ldots , d_{j}$ be $S$-monomials from $R(m_{n+1})$ such that $v_{1}, \ldots , v_{\beta}\notin \sum_{k=1}^{n}E_{k}.$  Then there is a linear mapping $T':R(m_{n+1})\rightarrow R(m_{n+1})$ such that 
\[T'(\sum_{k=1}^{n}E_{n}\cap R(m_{n+1}))=0.\] Moreover there are nonzero $S$-monomials $d'_{1}\ldots d'_{j}\in \{ab, a(b^{2}-b)\}$ and nonzero $\alpha _{1}, \ldots , \alpha _{j}\in F$ and nonzero $d\in R(m_{n+1}-1)$ such that  $T'(v_{k})= \alpha _{k}\cdot dd'_{k}$ for all $k\leq \beta$. Moreover, there is a linear mapping 
$f:R(m_{n+1}-1)\rightarrow R(m_{n+1}-1)$ such that  $T'(uv)=f(u)v$ for all $S$-monomials $u,v$ with $v\in R(1)$ and $u\in R(m_{n+1}-1)$.
\end{lemma}
\begin{proof} By the definition  of sets $E_{k}$, $R(m_{i+1})\cap\sum_{k=1}^{i}E_{k}=P'\cdot R(1)$ for some linear space $P'$. Each $v_{k}$ can be written as $v_{k}=c_{k}e_{k}$ where $c_{k}, e_{k}$ are $S$-monomials, and where $e_{k}\in  R(1)$. It follows that $c_{k}\notin P'$ for each $k$.
 We apply Lemma \ref{specialmappings} for $L=R(m_{i+1}-1)$ and   $T=P'$ to get linear mapping $f:R(m_{i+1}-1)\rightarrow R(m_{i+1}-1)$ such that $f(c_{k})\neq 0$ for every $k$, and the image of $f$ has codimension $1$ in $R(m_{i+1}-1)$. There is  $d\neq 0$ such that $d\in Im(f)$, and  since $f$ has codimension one, then the image of $f$ is $F\cdot d$. Hence $f(c_{k})$ is a non-zero multiply of  $d$  for every $k$. Moreover by Lemma \ref{specialmappings} we have  $f(P')=0$. Therefore mapping $T'(uv)=f(u)v$ defined  for $S$- monomials $u,v$ with $v\in R(1)$, $u\in R(m_{i+1}-1)$ and extended by linearity to all elements of $R(m_{i+1})$ satisfies the thesis of our theorem.
\end{proof} 

\begin{definition}\label{wazna2}  
Let $T':R(m_{n+1})\rightarrow R(m_{n+1})$ be a mapping as in Lemma \ref{t}.  
For every $j$,  we can extend the mapping $T'$ to the mapping $T:R(j\cdot m_{n+1})\rightarrow R(j\cdot m_{n+1}) $
 by defining $T(w_{1}\cdots w_{j})=T'(w_{1})\cdots T'(w_{j})$  for $w_{1}, \ldots , w_{j}\in  R(m_{n+1}),$ and then extending it by linearity to all elements from $R(j\cdot m_{n+1})$.

Let $M$ be a matrix with entries $m_{i,j}$, by $T(M)$ we will denote the matrix 
with corresponding  entries equal to $T(m_{i,j}).$
\end{definition}

\begin{lemma}\label{ostatni} Let $n\geq 1$ be a natural number. Let $0<m_{1}< m_{2}< \ldots $ be such that each $m_{i}$ is a power of two, $2^{2^{m_{i}}}<m_{i+1}$ and $2^{2^{m_{i}}}$ divides $m_{i+1}$ for every $i$.  
 Suppose that $w_{j}b\notin \sum_{i=1}^{n}E_{i}$ for every $j$, and $w_{t}b\in \sum_{i=1}^{n+1}E_{i}$ for some $t$.
 Denote $\beta =
2^{2^{m_{n+1}}}/m_{n+1}$ and  $k=2^{m_{n+1}}$. Then the following holds:
\begin{itemize}
\item  Then $w_{k+1}b=v_{1}\ldots v_{\beta }$ for some $S$-monomials $v_{1}, \ldots , v_{\beta }\in R(m_{n+1})$.
\item  Moreover,  there is a mapping  $T': R(m_{n+1})\rightarrow R(m_{n+1})$ satysfying assumptions of Lemma \ref{t} and such that $T'(v_{i})\neq 0$ for all $i\leq m_{n+2}/m_{n+1}$ and $T'(\sum_{i=1}^{n}E_{i}\cap R(m_{n+1}))=0$.
\item  Let 
 $T:R(m_{n+2})\rightarrow R(m_{n+1})$ be  defined as in  Definition \ref{wazna2} using our mapping $T'$, and let \[M=T(X_{n+1}^{m_{n+1}})\] where $X_{n+1}$ is as in Theorem \ref{id}.
Then $T(w_{k+1}b)\in P(M^{\beta })$.
\end{itemize}
\end{lemma}
\begin{proof} Observe that we can assume that $t$ is arbitrarily large, since $w_{i}b\in  \sum_{i=1}^{\infty }E_{i}$ implies $w_{i+1}b\in  \sum_{i=1}^{\infty }E_{i}$ by the definition of words $w_{i}$. 
 Therefore we can  assume that $t>m_{n+2}$.   
 By the definition of words $w_{i}b=W_{i}(a, b, -a, b^{2}-b)b$ and by Lemma \ref{one} we see that
 $w_{t}b=u_{1}\ldots u_{l}$ with each $u_{i}\in  R(m_{n+2})$ (since lenght of $w_{t}b$ is $2^{t}$ we can do it since $2^{t}\geq m_{n+2}$). 
 Observe that since $w_{t}b\in \sum_{i=1}^{n+1}E_{i}$ it follows that  for some $\xi $ we have  \[u_{\xi }\in \sum_{i=1}^{n+1}E_{i}\cap  R_{m_{n+2}}\]
(it can be proved  using linear mappings similarly as in the proof of  Lemma \ref{podslowa}).
   By the definition of words $w_{i}=W_{i}(a, b, -a, b^{2}-b)$  it follows that  $u_{\xi }\in Z_{\gamma }$ where \[Z_{\gamma}=\{w_{\gamma }b, w_{\gamma }(b^{2}-b), \bar {w}_{\gamma }b, \bar {w}_{\gamma }(b^{2}-b)\}\]
 where $2^{\gamma -1}=m_{n+2}$.  Similarly, \[u_{\xi }=v_{1}\ldots v_{m_{n+2}/m_{n+1}}\] for $v_{i}\in R(m_{n+1}).$ 

 Observe that for each $i$, $v_{i}\notin \sum_{i=1}^{n}E_{i}$ as otherwise
$w_{t}b\in  \sum_{i=1}^{n}E_{i}$ contradicting the inductive assumption. Therefore, by Lemma  \ref{t} applied to $S$-monomials $v_{i}$, we can choose mapping $T': R(m_{n+1})\rightarrow R(m_{n+1})$ such that $T'(v_{i})\neq 0$ for all $i\leq m_{n+2}/m_{n+1}$ and $T'(\sum_{i=1}^{n}E_{i}\cap R(m_{n+1}))=0$. 
 Let $T$ be as in Definition \ref{wazna2}.
 Observe that $T(\sum_{i=1}^{n}E_{i}\cap R(m_{n+2}))=0$ since   
 $T'(\sum_{i=1}^{n}E_{i}\cap R(m_{n+1}))=0$ (by Lemma \ref{t}). 
 It follows that \[T(u_{\xi })=T(v_{1}v_{2}\ldots v_{m_{n+2}/m_{n+1}})\in T(E_{n+1}).\]
  Notice that by the definition of the mapping $T$ we get \[T(v_{1})T(v_{2})\ldots T(v_{m_{n+2}/m_{n+1}})=T(v_{1}v_{2}\ldots v_{m_{n+2}/m_{n+1}})\in T(E_{n+1}).\]
 By the definition of $E_{n+1}$ it follows that \[T(v_{1}\ldots v_{\beta })\in T(E'_{n+1}),\]
 where $\beta =2^{2^{m_{n+1}}}/m_{n+1}$.
 Notice that $T(E'_{n+1})$ is the linear space spanned by coefficients of matrix $T({X_{n+1}}^{m_{n+1}\cdot \beta}).$
 Observe that $T(X_{n+1}^{m_{n+1}\cdot \beta })=T(X_{n+1}^{m_{n+1}})^{\beta }.$

 Therefore  \[T(v_{1})\ldots T(v_{\beta })=T(v_{1}\ldots v_{\beta })\in P(M^{\beta })\] where $M=T(X_{n+1}^{m_{n+1}})$.

Recall that $(v_{1}\ldots v_{\beta })\cdot (v_{\beta +1}\ldots v_{m_{n+1}/m_{n}})=u_{\xi }\in \{w_{\gamma }b, w_{\gamma }(b^{2}-b), \bar {w}_{\gamma }b, \bar {w}_{\gamma }(b^{2}-b)\}$, hence by the definition of words $w_{i}$ we get 
$v_{1}\ldots v_{\beta }=w_{k }b$ (since $\beta <m_{n+1}/m_{n}$).  
 Therefore $w_{k+1}b=v_{1}\ldots v_{\beta }$ and hence $T(w_{k+1}b)=T(v_{1}\ldots v_{\beta })\in P(M^{\beta })$.
\end{proof}

\begin{theorem}\label{ostatnia} Let $0<m_{1}< m_{2}< \ldots $ be such that each $m_{i}$ is a power of two, $2^{2^{m_{i}}}<m_{i+1}$ and $2^{2^{m_{i}}}$ divides $m_{i+1}$ for every $i$. Assume that $2^{m_{1}}>17!\cdot 10$. 
Let $R$, $I$ satisfy assumptions of Theorem \ref{id} for these $m_{i}$.
Let $R(i), S$ be as in Theorem \ref{id}. Let $n$ be a natural number.  
Then $w_{t}b\notin \sum_{i=1}^{n}E_{i}$ for any $t$. 
\end{theorem}
\begin{proof} 
 We proceed by induction on $n$. For $n=1$ the result holds by Theorem \ref{ia}. Suppose now that 
$w_{j}b\notin \sum_{i=1}^{n}E_{i}$ for any $j$. 
 We need to show that $w_{j}b\notin \sum_{i=1}^{n+1}E_{i}$ for all $j$. 
 Suppose on the contrary that $w_{t}b\in \sum_{i=1}^{n+1}E_{i}$ for some $t$, then by the definition of words $w_{i}$, 
 $w_{t+j}b\in \sum_{i=1}^{n+1}E_{i}$ for every $j\geq 0$.

 Let notation be as in Lemma \ref{ostatnia}, then $w_{k}b=v_{1}\cdots v_{\beta}$ for some $v_{1}, \ldots , v_{\beta }\in R(m_{n+1})$, moreover each $v_{i}$ is an $S$- monomial. 
  By Lemma \ref{ostatnia} 
  we know that $T(w_{k+1}b)\in P(M^{\beta })$.  

 Let $c'_{1}, c'_{2}, \ldots , c'_{j}\in \{v_{1}, \ldots , v_{\beta }\}$  be distinct $S$-monomials such that elements 
 $T'(c'_{1}),$ $ T'(c'_{2}), \ldots T'(c'_{j})$ form a basis of the linear 
space $F\cdot T'(v_{1})+\ldots + F\cdot T'(v_{\beta })$.  By Lemma \ref{ostatni}, for every $T(v_{k})=T(c'_{i_{k}})\alpha _{k}$ for some $i_{k}$ and some $0\neq \alpha_{k} \in F$. Denote $T(c'_{i})=c_{i}$, for every $i$.  Notice that $T(w_{k+1}b)= w'\cdot \alpha $ where $w'=c_{i_{1}}\cdots c_{i_{\beta }}$
 and $\alpha =\alpha _{1}\cdots \alpha _{\beta}$.
 It follows that $ c_{i_{1}}\cdots c_{i_{\beta}}\in P(M^{\beta })$. 

 Denote \[v=T'(v_{1})\ldots T'(v_{\beta /2 }), v'=T'(v_{\beta /2+1})\ldots T'(v_{\beta })\] 
and  \[u=T'(c'_{i_{1}})\ldots T'(c'_{i_{\beta /2 }}), u'=T'(c'_{i_{\beta /2+1}})\ldots T'(c'_{i_{\beta }}),\] then $ w'=c_{i_{1}}\cdots c_{i_{\beta }}=uu'$ and $u,u'\in R(2^{2^{m_{n+1}}}/2)=R(m_{n+1})^{\beta /2}$. 
 Moreover there are $0\neq \beta  ', \beta ''\in F$ such that $v=\beta '' \cdot  u$ and $v'= \beta '\cdot u'$.

 Notice that $u=c_{i_{1}}\ldots c_{i_{\beta /2}}$ and $u'=c_{i_{\beta /2+1}}c_{i_{\beta }}.$
By Lemma \ref{naj}   we get that 
for  $n'\geq 8d^{3}\cdot (\alpha +1)$ either $u$ or $u'$ has less than $n'$ subwords of length $n'$; without restraining generality we can assume that it happens for $u'$.
 In our case  $\alpha =(n+1)\cdot m_{n+1}$, $d=n+1$ because of assumptions on matrix $M=T(X_{n+1}^{m_{n+1}})$ (before Lemma \ref{naj}), so we can take   $n'=8(n+2)^{4}m_{n+1}$, so $v'$ has less than $n'$ subwords of degree $n'$.
  By Lemma \ref{lothaire} we get that  \[u'=cddd\ldots d\] for some words $c,d$ such that  
$c$ has length smaller than $n'!+n'!$ and $d$ has length $n'!$, so $d\in R(n'!\cdot m_{n+1})$ and $c\in R(l\cdot m_{n+1})$ for some $l<n'!+n'!$.

 It follows that 
\[v'=\beta ' \cdot cddd\ldots d.\]


 We know that $(v_{1}\ldots v_{\beta /2})\cdot (v_{\beta /2+1}\ldots v_{\beta})=w_{k+1}b=(w_{k}b)\cdot (\bar {w}_{k}b)$ therefore $v'=T(\bar {w}_{k}b)$.
  Observe now that $v'=T(\bar {w}_{k}b)$ implies \[v'=T(w_{k-1}(b^{2}-b)\bar {w}_{k-1}b).\]
 By Lemma \ref{dobry} we can write $w_{k-1}=ab_{2^{k-2}-1}\cdot b_{2^{k-2}-2}a\cdots ab_{2}\cdot ab_{1}a$ where $b_{i}\in \{b, b^{2}-b\}$.
 Then $v'=T(w_{k-1}(b^{2}-b)\bar {w}_{k-1}b)$ gives 
\[v'=T(ab_{2^{k-2}-1}\cdot ab_{2^{k-2}-2}\cdots ab_{2}\cdot ab_{1}\cdot a(b^{2}-b))(a\bar {b}_{1}\cdot a\bar {b}_{2}\cdots a\bar {b}_{2^{k-2}-1}\cdot ab)),\] where $\bar {b}=b^{2}-b$ and $\bar {b^{2}-b}=b$.

Recall that  $v'=  \beta '\cdot cddd\ldots d$ for some words $c,d$ such that  
$c$ has length smaller than $n'!+n'!$ and $d$ has length $n'!$ where $n'=8(n+2)^{4}m_{n+1}$.
By the assumptions  $v'\in R(2^{2^{m_{n+1}}}/2)=R(m_{n+1})^{\beta /2}$. 
Observe that $\beta /2> n'!\cdot 10\cdot m_{n+1}$.  Therefore we can write
\[\bar {w}_{k-1}(b^{2}-b){w_{k-1}}b=s\cdot p\cdot q\cdot r\cdot t\] for some $s,t\in R$ and some $p, q,r\in R(n'!\cdot m_{n+1})$ with
 $s\cdot p\cdot q=w_{k-1}(b^{2}-b)$. 

Notice that then $s\in R(q)$ for some $q>n'!\cdot 2$.
 Then \[v'=T({w}_{k-1}(b^{2}-b){\bar {w}_{k-1}}b)=T(s\cdot p\cdot q\cdot r \cdot t).\]

Notice that $spq=w_{k-1}(b^{2}-b)$ implies $spq\in R(\gamma )$ for some $\gamma $ divisible by $m_{n+1}$, as $m_{n+1}$ is a power of two and $w_{k-1}(b^{2}-b)\in R(2^{2^{m_{n+1}}}/4).$
 By the definition of $T$, 
 $v'=T(spq)T(qr)$. Because $p, q,r\in R(m!\cdot m_{j+1})$ we get 
\[v'=T(s)T(p)T(q)T(r)T(t).\]
 
 Then because $v'=cddd\ldots \cdot \beta '$ and $d$ has length $n'!$ we get $\gamma \cdot T(p)= \gamma ' \cdot T(q) =
 \gamma '' \cdot T(r) =m$ for some $0\neq \gamma , \gamma ' , \gamma '' \in F$ and some $m$ (where $m$ is a product of some  elements 
$c_{i_{1}},\ldots , c_{i_{\beta}})$.
 Denote $\xi =n'!\cdot m_{j+1}$, then $q=ab_{\xi -1}a\ldots b_{2}ab_{1}a(b^{2}-b)$, $r=a\bar {b}_{1}a\bar {b}_{2}a\ldots a\bar {b}_{\xi}$, and $p=p'ab_{\xi}$,  for some $p'$ (for some $b_{i}\in \{b, b^{2}-b\}$ where $\bar {b}=b^{2}-b$ and
 $\bar {b^{2}-b}=b$).  Recall that  $\gamma \cdot T(p)=\gamma '' \cdot T(r)$.
 Recall also that by Lemma \ref{t} we have 
 $T(p)=s\cdot ab_{\xi }$ and $T(r)=s'\cdot a \bar {b}_{\xi}$ for some $s,s'$
; it follows that 
$\bar {b}_{\xi }=b_{\xi}$, which is impossible, as $b_{\xi}\in \{b, b^{2}-b\}$ and  $\bar {b}=b^{2}-b$ and $\bar {b^{2}-b}=b$. We have obtained a contradiction. 

\end{proof}
\begin{theorem}\label{cesareo123}
 There is a nil ring $R$ such that the adjoint group of $R^{o}$ is not an Engel group.
 Moreover $R$  can be taken to be an algebra over an arbitrary countable field.
\end{theorem}
\begin{proof} Suppose first that $F$ is an infinite field.
Let $0<m_{1}< m_{2}< \ldots $ be such that each $m_{i}$ is a power of two, $2^{2^{m_{i}}}<m_{i+1}$ and $2^{2^{m_{i}}}$ divides $m_{i+1}$ for every $i$. Assume moreover that $2^{m_{1}}>17!\cdot 10$. 
Let $R$, $I$ satisfy assumptions of Theorem \ref{id} for these $m_{i}$.
 By Theorem \ref{ostatnia} we have  $w_{n}b\notin \sum_{i=1}^{j}E_{i}$ for any $n,j$,  where $w_{n}=W_{n}(a, b, -a, b^{2}-b)$.
 By Theorem \ref{id} we see that $I\subseteq \sum_{i=1}^{\infty }E_{i}+bE_{i} +b^{2}E_{i}+<a^{2}>+<b^{3}>.$
 Suppose that $w_{n}b\in I$ for some $n$, then there is $w\in \sum_{i=1}^{\infty }E_{i}$ such that 
 $w_{n}b- w\in bR+<a^{2}>+<b^{3}>$. We can assume that $ w\in \sum_{i=1}^{\infty }R(i)+R(i)a$ where $R_{i}=S^{i}$ and $S=F\cdot ab+F\cdot ab^{2}$.  
Observe that $(bR+<a^{2}>+<b^{3}>)\cap  (\sum_{i=1}^{\infty }R(i)+R(i)a)=0$ and so
$w_{n}b-w=0$ so $w_{n}b\in  \sum_{i=1}^{\infty }E_{i}$, a contradiction. It follows that $w_{n}b\notin I$ for every $n$.

 By  Lemma \ref{wazny} we have $v_{n}-1\notin I'$ for every $n$, and hence $(R/I)^{\circ }$ is not an Engel group ($I'$ is as in Lemma \ref{wazny}).  
 By Theorem \ref{id}, $R/I$ is nil. Therefore $R/I$ is a nil algebra over field $F$ such that the adjoint group 
$R/I^{o}$ of this algebra is not an Engel group. 

 If $F$ is a finite field then we proceed in the following way: 
 Let $\bar {F}$ be the algebraic closure of $F$, then $\bar {F}$ is infinite. Hence there is a nil algebra $A$ over $\bar {F}$ such that the adjoint group $A^{o}$ is not nil. Let $x,y\in A$ be such that $[x[\ldots [x[x,y]]]] \neq  1 $ ($n$ copies of $x$)
 for every $n$. Let $A'$ be the smallest  subring of $A$ containing  $x,y$ and such that  such that if $r\in A'$ then $f\cdot r\in A'$ 
for every $f\in F$. Then $A'$ is an $F$-algebra which is generated as $F$-algebra by $x,y$ and which is nil. Since $x,y\in A'$ then the adjoint group $A'^{o}$ is not an Engel group. 
\end{proof}
$ $

{\bf Proof of Corollary \ref{braces}.} Let $(R, +,\cdot )$ be a ring such that $R$ is nil and $R^{o}$ is not an Engel group. Assume moreover that $R$ is an algebra over a finite field of cardinality $p$ for some prime number $p$.  Let $(R, +,o)$ be the associated brace, so $a\circ b=a\cdot b+a+b$ for all $a,b\in R$, and the addition is the same as in the ring $R$. It follows that  
 $(R, +,o)$ satisfies the thesis of Corollary \ref{braces}.

\section{Zelmanov's question}\label{k}

 Observe that in the case of algebras over uncountable fields we have the following result analogous to Lemma \ref{pierwszy}.

\begin{lemma}\label{aggi} Let $F$ be an uncountable field and let $F'$ be a countable subfield of $F$.  
 Let $R$ be an $F$-algebra generated by elements $a,b$, and suppose that $a^{2}=0$ and $b^{3}=0$.
  Let $R[x_{1}, x_{2}, \ldots ]$ be the polynomial ring over $R$ in  infinitely many commuting variables $x_{1}, x_{2}, \ldots $.
Let  $F'[x_{1}, x_{2}, \ldots ]$ be the polynomial ring over $F'$ in  infinitely many commuting variables $x_{1}, x_{2}, \ldots $.
 
Let $S'$ be the $F'$-linear space spanned by elements $abx_{i}$ and  $ab^{2}x_{i}$ for $0\leq i$.
 
If all finite  matrices  with entries from $S'$ are nilpotent then $R$ is a Jacobson radical  ring. 
\end{lemma}
\begin{proof}
 Observe that if all finite  matrices  with entries from $S$ are nilpotent, then after substituting arbitrary elements $\alpha _{1} , \alpha _{2} ,\ldots \in F$ for variables $x_{1}, x_{2}, \ldots $ we get that all matrices 
  with entries from $F'$-linear space spanned by elements $ab^{i}\alpha _{j}$ are nilpotent.
 Therefore every  matrix with entries in the $F$-linear space spanned by elements $ab$ and $ab^{2}$ is nilpotent.
 By Lemma \ref{pierwszy},  $R$ is a Jacobson radical $F$-algebra.
\end{proof}

We recall Amitsur's theorem.
\begin{theorem}\label{amit}
 Let $R$ be a finitely generated algebra over an uncountable field. If $R$ is a Jacobson radical algebra then $R$ is nil.
\end{theorem}

We will now use Lemma \ref{aggi} to give  an analogon of Theorem \ref{id}.

Let $F$ be a field and let $F'$ be a countable subfield of $F$. Let $R$  be as in Lemma \ref{aggi} and let $S$ be 
the $F'$-linear space spanned by elements $ab$ and $ab^{2}$. Let $X=\{x_{1}, x_{2}, \ldots \}$ and let  $F'[X]$ denote the polynomial ring over $F'$ in infinitely many variables $x_{1}, x_{2}, \ldots $. 

 We can ennumerate all finite  matrices  with entries from $S\cdot F'[X]$ as 
   $X_{1}, X_{2}, \ldots $. We can assume that the matrix $X_{i}$  is a $d_{i}$-by-$d_{i}$ matrix where $d_{i}\leq i$  and $X_{i}$ has entries in
 $S\cdot y_{1}+S\cdot y_{2} +\ldots S\cdot y_{i}$ for some $y_{1}, y_{2}, \ldots , y_{i}\in F'[X]$  (if necessary taking $X_{i}=0$ for some $i$). The following  theorem has the same proof as Theorem  \ref{id}.
 
\begin{theorem}\label{id2} Let $F$ be an uncountable field,   and let $R, S$ and the  matrices $X_{1}, X_{2}, \ldots $ be as above.  Let $0<m_{1}<m_{2}< \ldots $ be a sequence of natural numbers such that $2^{2^{m_{i}}}$ divides $m_{i+1}$ for every $i\geq 1$. Denote $R(m)=F\cdot S^{m}$ for every $m$. 
Let $E'_{i}$ be the linear space spanned by all coefficients of polynomials which are entries of the matrix $X_{i}^{2^{2^{m_{i}}}}$ and  let
\[E_{i}=\sum_{j=0}^{\infty } R(j\cdot m_{i+1})E'_{i}SR.\]
  Then there is an ideal $I$ in $R$ contained in $\sum_{i=1}^{\infty }E_{i}+bE_{i} + b^{2}E_{i}+<a^{2}>+<b^{3}>$   and such that 
$R/I$ is a nil ring, where $<a^{2}>, <b^{3}>$ denote ideals in $R$ generated by elements $a^{2}$ and $b^{3}$.  
\end{theorem}
\begin{proof}  Observe first that the ideal $I_{k}$ of $R$ generated by coefficients of polynomials which are entries of the matrices  
$X_{k}^{2m_{k+1}}$  is contained in the subspace $E_{k}+bE_{k}+b^{2}E_{k}$.
  It follows because entries of every matrix $X_{k}$ have degree one in the subring generated by $S$ with elements of $S$ of degree one. In general if $n>m_{k+1}+2^{2^{m_{k}}}+1$ then  every entry of matrix
 $X_{k}^{n}$ belongs to $R(i)E'_{k}R(1)R$ for every $0\leq i<n-m_{k+1}-1$.
 
Define $I=\sum_{i=1}^{\infty }I_{k}+<a^{2}>+<b^{3}>$, then 
$I\subseteq \sum_{i=1}^{\infty }E_{i}+bE_{i} + b^{2}E_{i}+ 
<a^{2}>+<b^{3}>.$
 Observe also that, by Lemma \ref{aggi} and Theorem  \ref{amit}, $R/I$ is a nil ring.
\end{proof}

We will say that  $M, R, F', S, r_{1}, r_{2}, m, d,\alpha $ satisfy {\bf Assumption 3} if 
\begin{itemize}
\item[1.]   $R, F'$ are as in Lemma \ref{aggi} and  $S$ is 
the $F'$-linear space spanned by elements $ab$ and $ab^{2}$, and  $m, d, \alpha $ are natural numbers.
\item[2.]  $M$ is a $d$ by $d$ matrix with entries from $S^{m}\cdot F[X]$.
 Moreover, \[M\subseteq R+R\cdot y_{1}+R\cdot y_{2}+\ldots R\cdot y_{\alpha },\]
 for some $y_{1}, y_{2}, \ldots , y_{\alpha }\in  F[X],$ where $X=\{x_{1}, x_{2}, \ldots \}$ is an infinite set.
\end{itemize}

We now propose a generalisation of Lemma  \ref{naj}.

\begin{lemma}\label{najj} Let $F$ be an infinite field.  $M, R, S, m, d, \alpha $ satisfy Assumption $3$. Let $q$ be natural number.
 Let $c_{1}, \ldots , c_{k}$ be linearly independent  elements from $F\cdot S^{m}$, and let $r_{1}, r_{2}, \ldots , r_{\alpha +1}$ be products of 
$q$ elements from the set $C=\{c_{1}\ldots , c_{k}\}$. 
 If $n>d^{\alpha +2}\cdot (\alpha+1)^{\alpha +1}$ and for each $i$, $r_{i}$ has more than $n$ subwords of length $n$ then
 $r_{1}\cdot r_{2}\cdots r_{\alpha +1}\notin P(M^{j})$, for any $j$. 
 
We say that $w$ is a subword of degree $n$ in $r$, if $w$ is a product of $n$ elements from $C$, and $r=vwv'$ for some $v,v'$ which  are also products of elements from $C$.
\end{lemma}
\begin{proof} Denote $r=r_{1}\cdots r_{\alpha +1}$. Suppose on the contrary that $r\in P(M^{j})$. 
Let $p_{1, i}, \ldots , p_{n, i}$ be subwords of degree $n$ in $r_{i}$ for all $i\leq \alpha +1$.
  Fix numbers $\beta (1), \ldots , \beta (\alpha +1)\leq \alpha +1$, then there are $q_{1}, \ldots , q_{\alpha }$ such that 
\[s=(\prod_{i=1}^{\alpha }p_{\beta (i), i}q_{i} )p_{\beta (\alpha +1),\alpha +1}\] is a subword of $r$. 
 
 By Lemma \ref{podslowa}, $r\in P(M^{j})$ implies $s\in P(M^{j'})$, for some $j'$.
 
Let $f:F\cdot S^{m}\rightarrow F$ be a linear map 
 such that $f(c_{i})\neq 0$, and let $f:F\cdot S^{m\cdot n'}\rightarrow F$ be 
as in Definition \ref{wazna1}. Then  $f(c_{i_{1}}c_{i_{2}}\ldots c_{i_{n'}})=f(c_{1})\ldots f(c_{n'})\neq 0$ for every
 choice of $i_{1}, i_{2}, \ldots , i_{n'}\leq k$. 
By  Lemma \ref{mily} applied several times we get that
      
\[(\prod_{i=1}^{\alpha }p_{\beta (i), i}f(q_{i}) )p_{\beta (\alpha +1),\alpha +1}\in P((\prod_{i=1}^{\alpha } M^{n}f(M^{\deg q_{i}})M^{n}).\] By  an analogous argument to Lemma \ref{piaty}, 

\[p_{\beta (1),1}p_{\beta (2),2}\cdots p_{\beta (\alpha +1), \alpha +1}\in P((\prod_{i=1}^{\alpha } M^{n}Q_{i})M^{n}), \]
 where $Q_{i}=\sum_{i=1}^{d+1} F\cdot  f(M^{i})$.

Notice that the linear space $P((\prod_{i=1}^{\alpha } M^{n}Q_{i})M^{n})$
 has dimension at most $d^{\alpha }\cdot d^{2}\cdot ((\alpha+1)\cdot n)^{\alpha +1}$.

Observe that since each $p_{i, j}$ is a product of $n$ elements from the set $C$, then
 elements 
\[p_{\beta (1),1}p_{\beta (2),2}\cdots p_{\beta (\alpha +1), \alpha +1}\] span  linear space over field $F$ of dimension at least $n^{\alpha +1}$. Hence $d^{\alpha }\cdot d^{2}\cdot ((\alpha+1)\cdot n)^{\alpha }<n^{\alpha +1}$,
  a contradiction with the assumptions on $n$.
\end{proof}

{\bf Proof of Theorem \ref{cesareo}.} 
 We first obtain an analogon of Theorem \ref{ostatnia} by using Assumption $3$ instead of Assumption $2$ and by using  Theorem \ref{id}  instead of Theorem \ref{id2}, and by using Lemma \ref{najj} instead of Lemma \ref{naj}. Moreover, we 
  need to use  a stronger assumption that  $2^{2^{2^{2^{m_{i}}}}}$ divides $m_{i+1}$ (instead of the assumption that $2^{2^m_{i}} $ divides $m_{i+1}$). Once  we have obtained an analogon of Theorem \ref{ostatnia}, we can prove Theorem \ref{cesareo}  using the same proof as the proof of Theorem \ref{cesareo123}  where  instead of  Lemma \ref{naj} we use Lemma \ref{najj}. 

\section*{Acknowledgements} 

I would like to thank Efim Zelmanov and Gunnar Traustason for their helpful  comments on Engel groups, and  Tatyana Gateva-Ivanova for introducing me to research into braided groups, braces  and solutions to the Yang-Baxter equation.
 I am also grateful to Jan Krempa and Jan Okni{\' n}ski for their helpful suggestions. I am very grateful to Ferran Ced{\' o} and to the unknown referee for reading the manuscript and for their very useful comments.

\bibliographystyle{amsplain}

\begin{thebibliography}{99}
\bibitem{a} Alireza Abdollahi, {\em Engel elements in groups}, in Groups St Andrews 2009 in Bath,  London Math. Soc. Lecture Note Ser. 387,  Volume 1, Cambridge University Press, 2011, 94-117.
\bibitem{jain} Adel Alahmadi, Hamed Alsulami, S.K. Jain, 
Efim Zelmanov, {\em Finite generation of Lie algebras associated
with associative algebras}, J.  Algebra {\bf 426} (2015), 69--78.
\bibitem{ads} B. Amberg, O. Dickenschied, Ya.P. Sysak,  {\em Subgroups of the adjoint group of a radical ring}, Canad. J. Math. {\bf 50} (1998), 3--15.
\bibitem{4} B. Amberg, L. Kazarin, {\em Nilpotent p-algebras and factorized p-groups}, in  Groups St Andrews 2005. 
 London Math. Soc. Lecture Note Ser. 339, Volume 1, Cambridge University Press, 2007, 130--147.
\bibitem{6} B. Amberg, Ya.P. Sysak, {\em Radical rings and products of groups}, in  Groups St Andrews 1997 in Bath. London Math. Soc. Lecture Note Ser. 260,  Cambridge University Press, 2007,  1--19.
\bibitem{3} \bysame {\em  Radical rings with soluble adjoint group}, J. Algebra {\bf 247} (2002), 692-702. 
\bibitem{1} \bysame {\em Radical Rings with Engel Conditions}, J. Algebra {\bf 231} (2000), 36--373.
\bibitem{aw} J.C. Ault, J.F. Watters, {\em Circle groups of nilpotent rings},  Am. Math. Mon. {\bf 80} (1973), 48--52.
\bibitem{db} David Bachiller, {\em Counterexample to a conjecture about braces}, J. Algebra {\bf 453} (2016), 160--176.
\bibitem{dbs} \bysame {\em Extensions, matched products, and simple braces}, arXiv:1511.08477v3 [math.GR], 13 June 2016.
\bibitem{bc} David Bachiller, Ferran Ced{\' o}, {\em A family of solutions of the Yang-Baxter equation}, J. Algebra {\bf 412} (2014), 218--229.
\bibitem{bcj} David Bachiller, Ferran  Ced{\' o}, Eric Jespers, {\em Solutions of the Yang-Baxter equation associated with a left brace}, J. Algebra {\bf 463}  (2016), 80--102.
\bibitem{simple} D. Bachiller, F. Ced{\' o}, E. Jespers, J. Okni{\' n}ski, {\em Iterated matched products of finite braces and simplicity; new solutions of the Yang-Baxter equation}, 
 arXiv:1610.00477v1 [math.GR], 3 October 2016.
\bibitem{5}  F. Catino, M. M. Miccolia,  Ya. P. Sysak, {\em On the Adjoint Group of Semiprime Rings}, Comm. Alg. {\bf 35} (2006), 265--270. 
\bibitem{cr} F. Catino, R. Rizzo, {\em Regular subgroups of the affine group and radical
circle algebras}, Bull. Aust. Math. Soc. {\bf 79} (2009), 103--107.
\bibitem{cedo} Ferran Ced{\' o}, {\em  Braces and the Yang-Baxter equation},
(joint work with Eric Jespers and Jan Okni{\' n}ski), slides from the talk in 
Stuttgart, June 2012.
\bibitem{GIS} F. Ced{\' o}, T. Gateva-Ivanova, A. Smoktunowicz, {\em On the Yang-Baxter equation and left nilpotent left braces},  J. Pure  Appl Algebra  {\bf 221} (2017), no. 4, 751--756.
\bibitem{cjo} Ferran Ced{\' o}, Eric Jespers, Jan Okni{\' n}ski, {\em Braces and the Yang-Baxter Equation}, Comm. Math. Phys.
 {\bf 327} (2014), 101--116.
\bibitem{retractable} \bysame {\em Retractability of set theoretic solutions of the Yang-Baxter equation}, Adv. Math. {\bf 224} (2010), 2472--2484. 
\bibitem{cjonil} \bysame {\em Nilpotent groups of class three and braces}, Publ. Mat. {\bf 60}
 (2016), no. 1, 55--79.
\bibitem{cjr} F. Ced{\' o}, E. Jespers and ´A. del Rio, {\em Involutive Yang-Baxter Groups}, Trans. Amer. Math. Soc. {\bf 362} (2010), 2541--2558.
\bibitem{etingof} P. Etingof, T. Schedler, T. Solovov, {\em A set theoretical solutions to the quantum Yang-Baxter equation,}
 Duke. Math. J. {\bf 100} (1999), 169-209.
\bibitem{Tatyana} Tatiana Gateva-Ivanova, {\em Set-theoretic solutions of the Yang-Baxter equation, Braces, and Symmetric groups}, arXiv:1507.02602v2 [math.QA], 31 August 2015.
\bibitem{tatyana2} \bysame
{\em{A combinatorial approach to the set-theoretic
solutions of the Yang-Baxter equation}},
 J.Math.Phys. {\bf 45} (2004), 3828--3858.
\bibitem{tatyana4} T. Gateva-Ivanova and S. Majid, {\em Matched pairs approach to set-theoretic
solutions of the Yang-Baxter equation}, J. Algebra {\bf 319} (2008), 1462--1529.
\bibitem{tatyana5} \bysame {\em Quantum spaces associated to multipermutation
solutions of level two}, Algebr. Represent. Theor. {\bf 14} (2011), no. 2, 341--376.
\bibitem{tatyana6} T. Gateva-Ivanova, Peter Cameron, {\em Multipermutation solutions of the Yang-Baxter equation}, Comm. Math. Phys. {\bf 309} (2012), 583--621.
\bibitem{tatyana7} T. Gateva-Ivanova and M. Van den Bergh, {\em Semigroups of I-type}, J. Algebra
  {\bf 206} (1998), 97--112.
\bibitem{lv} L. Guarnieri, L. Vendramin, {\em Skew braces and the Yang-Baxter equation},  Math.Comp. published electronically, November 28, 2016. DOI: https://doi.org/10.1090/mcom/3161.
\bibitem{hal} Hales, A.W., Passi, I.B.S.: The second augmentation quotient of an integral group ring. Arch. Math.
(Basel) {\bf 31}  (1978), 259--265.  
\bibitem{natalia} Natalia Iyudu, Stanislav Shkarin, {\em Finite dimensional semigroup quadratic algebras with the minimal number of relations}, Monatshefte für Mathematik, {\bf 168} (2012), no. 2, 239--252.
\bibitem{jespers1} E. Jespers and J. Okni{\' n}ski, {\em Monoids and groups of I-type}, Algebr. Represent.
Theor. {\bf 8} (2005), 709--729.
\bibitem{jespers2} \bysame Noetherian Semigroup Algebras, Springer, Dordrecht,
2007.
\bibitem{lam} T.Y. Lam,  A First Course in Noncommutative Rings,  Springer-Verlag, New York, 1991.
\bibitem{LZY} J. Lu, M. Yan and Y. Zhu, {\em On the set-theoretical Yang-Baxter equation}, Duke. Math. J. {\bf 104} (2000), 1-18. 
\bibitem{lothaire} M. Lothaire, Algebraic Combinatorics on Words, 
 E-Book, DOI: http://dx.doi.org/10.1017/CBO9781107326019.

\bibitem{psz} V. M. Petrogradsky, I.P.  Shestakov, E. Zelmanov, {\em Nil graded self-similar algebras},
Groups Geom. Dyn. {\bf 4} (2010),  no. 4, 873-900.

\bibitem{plotkin} Boris Plotkin, {\em Notes on Engel groups and Engel elements in
groups. Some generalizations}, Izv. Ural. Gos. Univ. Mat. Mekh.  {\bf 36}  (2005), no. 7, 153--166, 192--193.
\bibitem{robinson} D. J. S. Robinson,  A course on group theory, Springer-Verlag, GTM Vol.
80, 1996.
\bibitem{rob} \bysame  Finiteness Conditions and Generalized Soluble Groups. Springer-Verlag, Berlin-Heidelberg-New
York, 1972.
\bibitem{rowen} Louis Halle Rowen, Graduate Algebra: Noncommutative View, Graduate Studies in Mathematics, Volume 91, American Mathematical Society, Providence, Rhode Island, 2008.
\bibitem{mob} Wolfgang Rump, {\em Modules over braces}, Algebra Discrete Math. {\bf 2} (2006), 127-137.
\bibitem{rum3} \bysame {\em The brace of a classical group}, Note Mat. {\bf 34} (2014), no. 1, 115-144.
\bibitem{rump} \bysame {\em Braces, radical rings, and the quantum Yang-Baxter equation},  J. Algebra {\bf 307} (2007), no. 1, 153--170.
\bibitem{52} \bysame {\em A decomposition theorem for square-free unitary solutions of the Yang-Baxter equation}, Adv. Math. {\bf 101} (1990), no. 3, 583--591.
\bibitem{shalev} A. Shalev, {\em On associative algebras satisfying the Engel condition}, Israel J.Math. {\bf 67} (1989), 287--290.
\bibitem{st} L. Sorkatti,  G. Traustason, {\em  Nilpotent symplectic alternating algebras I}, J.Algebra {\bf 423} (2015), 615--635. 
\bibitem{sz} I. P. Shestakov and E. Zelmanov, {\em Some examples of nil Lie algebras}, J. Eur. Math. Soc. (JEMS) 10 (2008), no. 2, 391--398.
\bibitem{amitsur} Agata Smoktunowicz, {\em How far can we go with Amitsur's theorem in differential polynomial rings}, to appear in Israel J. Math.
\bibitem{smok} \bysame {\em The Jacobson radical of rings with nilpotent homogeneous elements}, 
Bull. London Math. Soc. {\bf 40} (2008), no. 6, 917--928.
\bibitem{2} Ya. P. Sysak, {\em The adjoint group of radical rings and related questions}, Ischia Group Theory 2010: Proceedings of the Conference. Edited 
 by Patrizia Longobardi, Mariagrazia Bianchi, Mercede Maj, 344--365.
\bibitem{sysak2} \bysame {\em Products of groups and local nearrings},  Note. Mat. {\bf 28} (2008),  no. 2, 177--214.
\bibitem{ttt}  A. Tortora, M. Tota and G. Traustason, {\em Symplectic alternating nil-algebras}, J.
Algebra {\bf 357} (2012), 183--202.
\bibitem{t} G. Traustason, {\em  Engel groups}, in Groups St Andrews 2009 in Bath, Volume II. London Math. Soc. Lecture Note Ser. 388, Cambridge University Press, Cambridge 2011, 520--550.
\bibitem {v} L. Vendramin, {\em Extensions of set-theoretic solutions of the Yang-Baxter equation and a conjecture of Gateva-Ivanova},  J. Pure  Appl Algebra {\bf 220} (2016), no. 5, 2064--2076.
\bibitem{wis} I. Wisliceny, {\em Konstruktion nilpotenter associativer}, Algebren mit wenig Relationen. Math. Nachr. {\bf 147}
(1990), 75--82.
\bibitem{z1} E. I. Zelmanov, {\em On some problems of group theory and Lie algebra}, Mat.
Sb. {\bf 66} (1990), no. 1, 159-167.
\bibitem{z2}  \bysame {\em The solution of the restricted Burnside problem for groups
of odd exponent}, Math. USSR. Izu. {\bf 36} (1991), no.1, 41--60.
\bibitem{z3} \bysame {\em The solution of the restricted Burnside problem for 2-
groups}, Mat. Sb. {\bf 182} (1991), no. 4, 568--592.
\end{thebibliography}

\end{document}